\documentclass[11pt,a4paper]{article}
\usepackage{amssymb,amsmath,amsthm}
\usepackage{extarrows}
\usepackage{array}
\usepackage{booktabs}
\usepackage{multirow}
\usepackage{cite}
\usepackage{geometry}
\usepackage{layouts}

\newcommand{\R}{\mathbb{R}}

\newtheorem{theorem}{Theorem}[section]
\newtheorem{lemma}[theorem]{Lemma}

\newtheorem{remark}{Remark}[section]
\theoremstyle{definition}

\newtheorem{claim}{Claim}
\numberwithin{equation}{section}

\begin{document}
  \title{\LARGE\bf{Semiclassical solutions for critical Schr\"{o}dinger--Poisson systems involving multiple competing potentials}}
\date{}
\author{Lingzheng Kong,
Haibo Chen\thanks{\small Corresponding author: E-mail: math\_klz@csu.edu.cn (L. Kong),  math\_chb@163.com (H. Chen).}\\
{\small School of Mathematics and Statistics, Central South University,}\\
{\small Changsha, Hunan 410083, P.R. China}}
\maketitle
\begin{center}
 \begin{minipage}{14cm}
\begin{abstract}
In this paper,  a class of Schr\"{o}dinger--Poisson system involving multiple competing potentials and  critical Sobolev exponent is considered.
Such a problem cannot be studied with the same argument of the nonlinear term with only a positive potential, because the weight potentials set $\{Q_i(x)|1\le i \le m\}$ contains  nonpositive, sign-changing, and nonnegative elements. By introducing the ground energy function and subtle analysis, we first prove the existence of ground state solution $v_\varepsilon$  in the semiclassical limit via the Nehari manifold and  concentration-compactness principle. Then we show that $v_\varepsilon$ converges to the ground state solution of the associated limiting problem and   concentrates at a concrete set characterized by the potentials. At the same time, some properties for the ground state solution are also studied. Moreover, a sufficient condition for the nonexistence of the ground state solution is obtained.

\end{abstract}
\end{minipage}
\end{center}
\vskip6mm
{\section{Introduction and statement of  results}\label{sec1}}
 \setcounter{equation}{0}
 \par
 In this paper,  we study the existence and concentration behavior of positive ground state solutions for  the following  critical Schr\"{o}dinger--Poisson system:
\begin{equation}\label{11}
\begin{cases}
  -\varepsilon^{2}\Delta u+h(x)\phi u+V(x)u=\sum\limits_{i=1}^{m} Q_i(x)|u|^{q_i-2}u+K(x)|u|^{4}u,&x\in\R^3,\\
  -\varepsilon^{2}\Delta\phi=4\pi h(x)u^2,&x\in\R^3,
\end{cases}\tag{$SP_\varepsilon$}
\end{equation}
where $\varepsilon>0$ is a small parameter, $4<q_1<q_2<\cdots<q_m<6$. $V(x)$ is the potential, $h(x)$ is the  electronic potential,  $Q_i(x)(i=1,\cdots,m)$ and $K(x)$ are weight potentials satisfying some competing conditions. Moreover, the nonlinear growth of $|u|^{4}u$ reaches the Sobolev critical exponent since $2^*=6$ for three spatial dimensions, which is why we call critical Schr\"{o}dinger--Poisson system in the title.
\par
Replacing the nonlinear term $\sum\limits_{i=1}^{m} Q_i(x)|u|^{q_i-2}u+K(x)|u|^{4}u$ by $f(x,u)$, \eqref{11} becomes the following system: 
\begin{equation}\label{112}
\begin{cases}
-\varepsilon^{2}\Delta u+h(x)\phi u+V(x)u=f(x,u),&x\in\R^3,\\
-\varepsilon^{2}\Delta\phi=4\pi h(x)u^2,&x\in\R^3,
\end{cases}
\end{equation}
which known as the nonlinear Schr\"{o}dinger--Maxwell system, arises in an interesting physical meaning because it appears in quantum mechanics models \cite{RHE,IPL,ET}  and in semiconductor theory \cite{VD2,PLL, PMCC}.
In  fact, systems like \eqref{112} have been first proposed in \cite{VD}  as a model describing the interaction between the electrostatic field and the solitary waves of nonlinear Schr\"{o}dinger type equation
\begin{equation}\label{111}
i\hbar \frac{\partial\psi }{\partial t}=-\frac{\hbar^2}{2m}\Delta\psi+W(x)\psi-|\psi|^{p-2}\psi,
\end{equation}
where $i$ is the imaginary unit, $\hbar$ is the Planck constant, $m$ is the mass of the field $\psi$, $W(x)$ is the time independent potential of the particle at the position $x\in\R^3$. Then, looking for the standing waves of Eq. \eqref{111}, namely waves of the form
$\psi(x,t)=u(x)e^{-i\omega t/\hbar},\hspace{1ex}x\in\R^3,\hspace{1ex}t\in\R,$
one is led to the system \eqref{112} with $\varepsilon^2=\frac{\hbar^2}{2m}$, $V(x)=W(x)-\omega$.
For more details on physical background, we refer the reader to \cite{VD,VD2} and the references therein.
\par 
In recent years, the qualitative analysis of positive solutions for problem \eqref{112}, including the existence, nonexistence, concentration behavior, and multiplicity,  has been widely investigated under various assumptions of the potentials. For more details, we refer the readers to previous studies  \cite{APA,APA2,VD,AD,PMCC,LFZ,DR,DR2,HLR,HLR2,JLJF,JLJF2,MY,PHR,Cerami,Cerami2,TXH,CST,LHL,SJT,CGF}. In particular, if $\varepsilon=V(x)=h(x)=1$, see, for example, Benci and Fortunato \cite{VD} for the case of $f(x,u)=0$,  demonstrated the existence of infinitely many solutions of an eigenvalue problem on bounded domain. 
Azzollini et al. \cite{APA}, when the nonlinearity satisfies Berestycki-Lions type assumptions, proved the existence of nontrivial solution.
\cite{AD,DR} for the pure power nonlinearity $|u|^{p-2}u$ with $p\in(2,6)$ by working in the subspace of radial functions of $H^1(\R^3)$ i.e., $H^1_r(\R^3)$, obtained multiple bound state solutions, the existence and nonexistence results, respectively.
Zhao and Zhao \cite{LFZ} for  the case of
$f(x,u)= \mu Q(x)|u|^{q-2}u+K(x)|u|^4u$ where $q\in[4,6)$ and $\mu>0$, introduced the assumption on the weight potential $K(x)$:
\vskip2mm
\par\noindent
$(K_1)$ $|K(x)-K(x_0)|=o(|x-x_0|^{\alpha}),\hspace{1ex}\text{where}\hspace{1ex} 1\leq\alpha<3\hspace{1ex} \text{and}\hspace{1ex} K(x_0)=\max\limits_{x\in \R^3}K(x),$
\vskip2mm
\par\noindent
to estimate the critical energy level, and proved the existence of a positive solution based on the methods of Brezis and Nirenberg and Lions' concentration-compactness principle. 

\par
If the electronic potential $h(x)$ is not a constant, 
Cerami and Vaira \cite{Cerami} for the case of $f(x,u)=a(x)|u|^{p-1}u$, $p\in(3,5)$ without any symmetry assumptions, by  using the Nehari manifold  and establishing compactness lemma, proved the existence of positive ground state and bound state solutions. 
Huang et al. \cite{HLR} for system \eqref{112} with $f(x,u)=a(x)|u|^{p-2}u+\mu K(x)u$, where $p\in (4,6)$, $K(x)\geq0$  and $a(x)$ can change sign without any symmetry assumptions, mainly prove the existence of at least two positive solutions via a comparison between $\mu$ and the first eigenvalue of $-\Delta u+id$. 
If $V(x)$ is not a constant, and, then $V(x)$ may not be radial, so one cannot work in $H^1_r(\R^3)$ directly. To overcome this difficulty, by assuming that $V(x)$ satisfying $$0<V_0=\inf\limits_{x\in\R^{3}}V(x)<V_\infty=\lim\limits_{|x|\to\infty}V(x),$$ which  was firstly introduced in \cite{PHR}. Azzollini et al. \cite{APA2} for the case of $f(x,u)=|u|^{p-1}u$, proved the existence of ground state solutions and they also studied the existence of solutions for the critical growth by concentration compactness principle. 
By assuming that $ V(x)$, $a(x)$ and $h(x)$ satisfy some  decay rates, Cerami and Molle \cite{Cerami2} for the case of $f(x,u)=a(x)|u|^{p-1}u$, proved the existence of a positive bound state solution  via the Nehari manifold, which complements the results given by \cite{Cerami} in some sense.
\par 
If the parameter $\varepsilon\to0$,  the systems like \eqref{112} and bound states are called semiclassical problems and semiclassical states, respectively. Semiclassical states can be used to describe a kind of transition between Quantum Mechanics and Newtonian Mechanics. 
He and Zou \cite{hezou} proved the existence of ground state solutions for the critical growth concentrating on the minima of $V(x)$. 
Ruiz \cite{DR2} showed the semiclassical states concentrating around a sphere when $V(x)$ statisfies some suitable assumptions. J. Wang et al. \cite{JLJF} for the case of $f(x,u)=b(x)g(u)+|u|^4u$,  proved that there are two families of positive solutions concentrating on the maxima of $b(x)$ and the minima of $V(x)$, respectively. 
Yang \cite{MY} for the case of $f(x,u)=P(x)g(u)+Q(x)|u|^4u$ where $\min\limits_{x\in\R^3}V(x)>0$, $\inf\limits_{x\in\R^3}P(x)>0$ and $\min\limits_{x\in\R^3}Q(x)>0$, proved the existence of semiclassical solutions concentrating at a special set characterized by the potentials, particularly, if $\mathcal{V}\cap \mathcal{P}\cap\mathcal{Q}\neq\emptyset $,  then the ground state solusion of $-\Delta v +V_{min}v=P_{max}g(v)+Q_{max}|v|^4v$ is obtained, where $\mathcal{V}=\{y\in\R^3: V(y)=\min\limits_{x\in\R^3}V(x)\}$, $\mathcal{P}=\{y\in\R^3: P(y)=\max\limits_{x\in\R^3}P(x)\}$ and $\mathcal{Q}=\{y\in\R^3: Q(y)=\max\limits_{x\in\R^3}Q(x)\}$. As one can see, the interaction of several potentials and the nonlinearity with critical growth have a significant impact on the existence and concentration behavior of positive solutions for the Schr\"{o}dinger--Poisson systems like \eqref{112}. However, the  existence and concentration results on them involving both the critical growth nonlinearity and more than four potentials have not previously been described. That is the main motivation of this work.
\par
Inspired by the fact mentioned above, the purpose of this paper is to study how the interaction of more than four potentials will make an impact on the existence and concentration of the ground state solutions for  the critical Schr\"{o}dinger--Poisson system \eqref{11} when $\varepsilon$ is small. It is quite natural to ask that: can we obtain a concentration results for the ground state solutions of \eqref{11} ? If so, where? In the present paper,  we shall give some answers for these questions. Furthermore, the aim of our work is twofold: (i) to study the existence of positive ground state solution for \eqref{11} and its properties, such as concentration, exponential decay, etc. (ii) to find some sufficient conditions  for the nonexistence of positive ground state solution.
\par
It is well known that system \eqref{11} can be easily transformed into a single Schr\"{o}dinger equation with a nonlocal term. In fact,  as we shall see in Section 2, for every $u\in H^{1}(\R^3)$, and any fixed $\varepsilon>0$,
 applying the Lax-Milgram theorem,
a unique $\phi_{u/\varepsilon}\in \mathcal{D}^{1,2}(\R^3)$ is obtained, such that $ -\varepsilon^{2}\Delta\phi=4\pi h(x)u^2$ and that, inserted into the first equation of \eqref{11},
gives
\begin{equation}\label{1}
-\varepsilon^{2}\Delta u+h(x)\phi_{u/\varepsilon} u+V(x)u=\sum\limits_{i=1}^{m} Q_i(x)|u|^{q_i-2}u+K(x)|u|^{4}u.
\tag{$S_{\varepsilon}$}
\end{equation}
Hence, $v$ is a solution of \eqref{1} if and only if $(v, \phi_{v/\varepsilon})$ is a solution of \eqref{11}. For simplicity, in many cases we say $v\in H^1(\R^3)$, instead of $(v, \phi_{v/\varepsilon})\in H^1(\R^3)\times \mathcal{D}^{1,2}(\R^3)$, is a weak solution of \eqref{11}.
Then in the following we only need to study the Eq. \eqref{1}.
\par
Before stating the main result of this paper, we introduce the precise assumptions on $h$, $V$, $K$, and $Q_i(i=1,2,\cdots,m)$ :
 \vskip2mm
 \par\noindent
$(f_1)$ $V(x),K(x),Q_i(x)(1\le i \le m)\in C^{1}(\R^{3},\R)$.\\
$(f_2)$ $(i-i_0)Q_i(x)\geq 0$ for $i \neq i_0$, $Q_{i_0}(x)$ is allowed to change sign.\\
$(f_3)$ $0<V_0:=\inf\limits_{x\in\R^{3}}V(x)$, $V(x)$ and $Q_i(x)(1\le i \le m)$ are bounded in $\R^3$.\\
$(f_4)$ $0\le K(x)\le K^{\infty }:=\limsup\limits_{|x|\to\infty }K(x)$, and there exists $x_0\in\R^{3}$ such that $K(x_0)=K^{\infty }$.\\
$(H)$ $h(x)\in C(\R^3,\R), 0< h_0=\inf\limits_{x\in\R^{3}}h(x)\le h(x)\le h_\infty:=\lim\limits_{|x|\to\infty }h(x)<\infty.$
\vskip2mm
\par
For every $s\in\R^3$, we consider the following equation with parameters:
\begin{equation}\label{gs}
-\Delta u+\left[\frac{1}{|x|}*u^2\right]h^2(s)u+V(s)u=\sum\limits_{i=1}^{m}Q_i(s)|u|^{q_i-2}u+K(s)|u|^{4}u,\quad x\in\R^3,
\end{equation}
where $s\in \R^3$ acts as a parameter instead of an independent variable.
The ground energy function $G(s)$, which is defined to be the ground energy associated with \eqref{gs}, and was firstly introduced in \cite{XB}. See section \ref{scr} for a more detailed. Moreover, let $c_\infty$ be the ground energy associated with ``limiting problem'' of \eqref{1}, which is given as
\begin{equation}\label{limpro}
-\Delta u+\left[\frac{1}{|x|}*u^2\right]h^2_{\infty}u+V_{\infty}u=\sum\limits_{i=1}^{m} Q_i^{\infty}|u|^{q_i-2}u+K^{\infty}|u|^{4}u,\quad x\in\R^3,
\end{equation}
where $
V_{\infty}:=\liminf\limits_{|x|\to\infty}V(x),\hspace{1ex} Q_i^{\infty}:=\limsup\limits_{|x|\to\infty}Q_i(x).
$
\par
Then we state our main results of this work as follows.
\begin{theorem}\label{the1}
 Suppose that the potentials  $V$, $K$, and $Q_i(1\le i\le m)$ satisfy conditions $(f_1)$--$(f_4)$, the electronic potential $h(x)\equiv1$ and 
 \begin{equation}\label{14}
 c_\infty>c_0:=\inf\limits_{s\in \R^3}G(s).
 \end{equation}
 Then, for $\varepsilon>0$ small enough,  \eqref{1} has a positive ground state solution $v_\varepsilon$. Moreover, we have
 	\begin{tabular}{@{}rp{15.05cm}}
 	\rm{(\romannumeral1)} & The positive ground state solution $v_\varepsilon $ possesses a maximum point $x_\varepsilon$ in $\R^3$, such that $\lim\limits_{\varepsilon\to 0}dist(x_\varepsilon,\mathcal{G})=0$. Setting $\eta_\varepsilon(x)=v_\varepsilon(\varepsilon x+x_\varepsilon)$, where $x_\varepsilon\to x_0,$ as $\varepsilon\to0$, and $\eta_\varepsilon$ converges in $H_\varepsilon$ to a positive ground state solution of\\
\rule[-15pt]{0pt}{35pt}&\multicolumn{1}{c}{$-\Delta u+\phi_u u+V(x_0)u=\sum\limits_{i=1}^{m} Q_i(x_0)|u|^{q_i-2}u+K(x_0)|u|^{4}u,\hspace{1ex}x\in\R^3,$}\\
\rule[0pt]{0pt}{0pt}& where $\mathcal{G} :=\{s\in\R^3;G(s)=c_0\}.$\\
 	\rm{(\romannumeral2)} &There exist constants $C>0$ and $\mu >0$ such that\\
\rule[20pt]{0pt}{0pt}	&\multicolumn{1}{c}{$v_\varepsilon(x)\le  C exp\left({-\frac{\mu}{\varepsilon}|x-x_\varepsilon|}\right)$.}\\
 \end{tabular}
\end{theorem}
\begin{remark}\label{rem1}
\rm{}The existence of solution can be obtained, if the the electronic potential $h$ satisfies condition $(H)$.
However, for concentration, it is complicated  since the nonlocal term
$$\int_{ \R^3}h(\varepsilon x)\left[\frac{1}{|x|}*\left(h(\varepsilon x)u^2\right)\right] u^2dx= \int_{\R^3}\int_{\R^3}h(\varepsilon x)u^2(x)\frac{h(\varepsilon y)u^2(y)}{|x-y|}dydx,$$ appears  $h(\varepsilon x)$ and $h(\varepsilon y)$. For the sake of simplicity, we assume that $h(x)\equiv1$.
\end{remark}
\begin{remark}
\rm{}Conditions like $(f_1)$--$(f_4)$ on the nonlinear term was introduced by \cite{FHN}. Fan obtained the existence of positive solutions of Kirchhoff-type problem.
Theorem \ref{the1} extends the main results in \cite{FHN} to the Schr\"{o}dinger--Poisson systems.
\end{remark}
Now we give our key idea for the proof of  Theorem \ref{the1}.  As we deal with the problem \eqref{1} in $H^1(\R^3)$, the sobolev embeddings $H^1(\R^3)\hookrightarrow L^q(\R^3)$, $q\in[2,6)$ are not compact. The energy functional does not satisfy $(PS)_c$ condition at any energy level $c$. To overcome this obstacle, we try to pull the energy level down below critical energy level. Different from the assumptions $(K_1)$ on $K(x)$ and the methods in \cite{LFZ,HLR2}, we first prove that the upper bound of the energy level defined in \eqref{lstene} no more than the infimum of $G(s)$. Then, under our conditions, we obtain the estimation of the critical level for the infimum of $G(s)$   with the help of technique of Brezis and Nirenberg \cite{{Bre}}. On the other hand, one can see that  the methods used in \cite{hezou,JLJF,JLJF2, MY} to establish a concentrating set of the ground state solutions for \eqref{1} does not work,  owing to the competing relationship among multiple potentials. We succeed in doing so by introducing a new concentration set $\mathcal{G}$ for the ground state solutions, which is not relies on the condition
\begin{equation}\label{15}
\mathcal{M}=\bigcap_{i=1}^{m}\mathcal{Q}_i\cap\mathcal{V} \cap \mathcal{K}\neq\emptyset,
\end{equation}
where $\mathcal{V}=\{y\in\R^3: V(y)=\min\limits_{x\in\R^3}V(x)\}$, $\mathcal{Q}_i=\{y\in\R^3: Q_i(y)=\max\limits_{x\in\R^3}Q_i(x)\}$ and $\mathcal{K}=\{y\in\R^3: K(y)=\max\limits_{x\in\R^3}K(x)\}$. Particularly, if the condition \eqref{15} holds, it is easy to see that $\mathcal{M}=\mathcal{G}$.
\par
To overcome the difficulty is mentioned in Remark \ref{rem1}, we assume that the electronic potential $h(x)$ satisfies
 \vskip2mm
\par\noindent
$(H_1)$ $h(x)\in C(\R^3,\R)$, $h(x)\geq0$, $\lim\limits_{|x|\to \infty }h(x)=0$ and $h(x)=0$ if $x\in \mathcal{G}.$ 
\vskip2mm
\par
Then we have the following result.
\begin{theorem}\label{the2}
	Suppose that the potentials  $V$, $K$, and $Q_i(1\le i\le m)$ satisfy conditions $(f_1)$--$(f_4)$, the electronic potential $h$ satisfies $(H_1)$ and 
	\begin{equation*}
	c_\infty>c_0:=\inf\limits_{s\in \R^3}G(s).
	\end{equation*}
	Then, for $\varepsilon>0$ small enough,  \eqref{1} has a positive ground state solution $v_\varepsilon$. Moreover, we have\\
\begin{tabular}{@{}rp{15.05cm}}
\rm{(\romannumeral1)} & The positive ground state solution $v_\varepsilon $ possesses a maximum point $x_\varepsilon$ in $\R^3$, such that $\lim\limits_{\varepsilon\to 0}dist(x_\varepsilon,\mathcal{G})=0$. Setting $\eta_\varepsilon(x)=v_\varepsilon(\varepsilon x+x_\varepsilon)$, where $x_\varepsilon\to x_0,$ as $\varepsilon\to0$, and $\eta_\varepsilon$ converges in $H_\varepsilon$ to a positive ground state solution of\\
\rule[-15pt]{0pt}{35pt}&\multicolumn{1}{c}{$-\Delta u+V(x_0)u=\sum\limits_{i=1}^{m} Q_i(x_0)|u|^{q_i-2}u+K(x_0)|u|^{4}u,\hspace{1ex}x\in\R^3,$}\\
&where $\mathcal{G}:=\{s\in\R^3;G(s)=c_0\}.$\\
\rm{(\romannumeral2)} &There exist constants $C>0$ and $\mu >0$ such that\\
\rule[-8pt]{0pt}{25pt}&\multicolumn{1}{c}{$v_\varepsilon(x)\le  C exp\left({-\frac{\mu}{\varepsilon}|x-x_\varepsilon|}\right).$}\\
	\end{tabular}
\end{theorem}
%\begin{remark}
%\rm{}The results of Theorem \ref{the1} and \ref{the2} consider \eqref{11} with $m\geq3$, which  generalize ones in \cite{LFZ,MY}. 
%\end{remark}
\begin{remark}
\rm{}It is interesting to give some sufficient conditions to guarantee \eqref{14}, in terms of $V$, $K$ and $Q_i(i=1,2,\cdots,m)$. For example, let us consider the following conditions:\\
	\begin{tabular}{@{}rp{15.05cm}}
\rm{(1)}&$V_\infty=\sup\limits_{x\in\R^3}V(x)$, $Q_i^\infty=\inf\limits_{x\in\R^3}Q_i(x)(i=1,2,\cdots,m)$, $K(x)\equiv C$.\\
\rm{(2)}& There exists $\hat{s}$ such that 
$$ V_\infty\geq V(\hat{s}),\hspace{1ex} Q_i^\infty \le Q_i(\hat{s})(i=1,2,\cdots,m), K(x)\equiv C$$
 with one of the above inequalities being strict.
\end{tabular}
\end{remark}
If $V$, $K$, and $Q_i(i=1,2,\cdots,m)$ are not all constants, then each of the previous conditions (1) and (2) guarantee $c_\infty>c_0$ (see\cite{SCM,XB}).
\par
Finally, to get the nonexistence of ground state solution, we make the following assumption:
\vskip2mm
\par\noindent
$(f_5)$ $V(x)\geq V_\infty =V_0,$ and $Q_i(x)\le Q_i^{\infty}(1\le i \le m).$ $h(x)\equiv1$ or $h(x)$ satisfies $(H_1)$.
\vskip2mm
\par
The result of this work has the following statement.
\begin{theorem}\label{the3}
	Suppose that conditions $(f_1)$--$(f_5)$ hold. Then, for any $\varepsilon>0$, \eqref{1} has no ground state solution.
\end{theorem}
\par
 The remainder of this paper is organized as follows. In Section \ref{sec2}, we derive a variational setting for the problem and give some preliminaries. In section \ref{sec3} and \ref{sec4}, we prove the existence of positive ground state solutions for \eqref{22} with some properties, such as concentration, exponential decay etc. The proofs of Theorems 1.1 and 1.2 will be given in Section \ref{sec5}.   Section \ref{sec6} is dedicated to the proof of Theorem \ref{the3}.
\par
 Hereafter we use the following notations:\par
 $\bullet$ $H^1(\R^3)$ is the usual Sobolev space equipped with the inner product abd norm
 $$(u.v)=\int_{\R^3}(\nabla u\nabla v)+uv)dx;\hspace{2ex}||u||^2=\int_{ \R^3}(|\nabla u|^2+u^2)dx.$$\par
 $\bullet$ $\mathcal{D}^{1,2}(\R^3) $ is the completion of $C^{\infty}_0(\R^3)$ with respect to the norm 
 $$||u||_{\mathcal{D}^{1,2}}^2=\int_{ \R^3}|\nabla u|^2dx.$$\par
 $\bullet$ $L^r(\Omega)$, $1\le r\le\infty$, $\Omega\subset\R^3$, denotes a Lebesgue space,  the norm in $L^r(\R^3)$ is denoted by\par \hspace{1.25ex} $|u|_{r,\Omega}$,  where $\Omega$ is a proper subset of $\R^3$, by $|u|_r$ where $\Omega=\R^3$.\par
 $\bullet$ Denote the best constants for the embeddings of $H^1(\R^3)\hookrightarrow L^p(\R^3)(2\le p\le6)$ and\par \hspace{1.25ex} $\mathcal{D}^{1,2}(\R^3)\hookrightarrow$ $ L^6(\R^3)$ by $S_p$ and $S$, respectively. Then 
 \begin{equation}\label{spine}
 |u|_p\le S_p^{-\frac{1}{2}}||u||\quad \forall u\in H^1(\R^3),
 \end{equation}\par \hspace{2ex}
 and
 \begin{equation}
 |u|_6\le S^{-\frac{1}{2}}||u||_{\mathcal{D}^{1,2}}\quad \forall u\in \mathcal{D}^{1,2}(\R^3).
 \end{equation}\par
 $\bullet$ For any $R>0$ and for any $z\in\R^3$, $B_R(z)$ denotes the ball of rasius $R$ centered at $z$.\par
 $\bullet$ $C$, $C_k(k=0,1,\cdots ,n+1)$ denote various positive constants may different from line to  line.\par
 $\bullet$ $\rightarrow$ and $\rightharpoonup$ denote the strong and weak convergence in the related function space respectively.\par
 $\bullet$ $o_n(1)$ denotes any quantity which tends to zero when $n\to\infty$.
 \vskip4mm
 {\section{Preliminary results and variational framework}\label{sec2}}
\setcounter{equation}{0}
\vskip2mm
In this section,   we outline the variational framework of problem \eqref{11} and give some preliminary lemmas. From this section to the section \ref{subsec51}, we assume that the electronic potential $h\equiv1$. 
\par
For every $u\in H^{1}(\R^3)$, the linear functional $L_u$ defined in $\mathcal{D}^{1,2}(\R^3)$ by
	\begin{equation}
\begin{aligned}
L_u(v)=\int_{ \R^3}u^2vdx.
\end{aligned}
\end{equation}
Then it follows from Lax-Milgram theorem that there exists a unique $\phi_u\in\mathcal{D}^{1,2}(\R^3)$ such that 
	\begin{equation}\label{pse}
\begin{aligned}
\int_{\R^3}\nabla \phi_u\nabla vdx=\int_{\R^3}u^2vdx,\hspace{1ex}  \forall v\in\mathcal{D}^{1,2}(\R^3),
\end{aligned}
\end{equation}
which is a weak solution of $- \Delta \phi=4\pi u^2$ and the following representation formula holds
\begin{equation*}
\phi_u(x)=\int_{\R^3}\frac{u^{2}(y)}{|x-y|}dy=\frac{1}{|x|}*u^2.
\end{equation*}
Therefore, for any $\varepsilon>0$, we have
\begin{equation}\label{sp}
\begin{aligned}
\phi_{u/\varepsilon}=\frac{1}{\varepsilon^2}\int_{\R^3}\frac{u^{2}(y)}{|x-y|}dy=\varepsilon^{-2}\phi_u.
\end{aligned}
\end{equation}
\par
Let us define the operator $\Phi$:$H^{1}(\R^3)\to\mathcal{D}^{1,2}(\R^3)$
 as
\begin{equation*}
\Phi[u]=\phi_u.
\end{equation*}
We next state  some  properties of $\Phi$, which will be useful in the following.
\begin{lemma}\label{L21}\quad \vspace{0.5ex}\\
\begin{tabular}{@{}rl}
\rm{(\romannumeral1)} &$\Phi$ is continuous;\\
\rm{(\romannumeral2)} &$\Phi$ maps bounded sets into bounded sets;\\
\rm{(\romannumeral3)} &$\Phi(tu) = t^{2}\Phi(u)$;\\
\rm{(\romannumeral4)} &$||\phi_{u}||_{\mathcal{D}^{1,2}}\le S^{-\frac{1}{2}}|u|_{12/5}^{2}$.\vspace{0.5ex}\\
	\end{tabular}
\end{lemma}
\begin{proof}
The proof of (i) and (ii) can be found in \cite{Cerami}. (iii) and (iv) are clear from the definition of $\phi_u$. Replacing $v$ by $\phi_{u}$ in \eqref{pse} and using H\"{o}lder inequality, we have 
$$||\phi_{u}||_{\mathcal{D}^{1,2}}^2=\int_{\R^3}\phi_{u}u^2 dx\le |\phi_{u}|_6 |u|_{12/5}^{2}\le S^{-\frac{1}{2}}||\phi_{u}||_{\mathcal{D}^{1,2}} |u|_{12/5}^{2},$$
and then (iv) holds.
\end{proof} 
\begin{lemma}\label{L22}
	Assume $u_n\rightharpoonup u \in H^{1}(\R^3) $.Then\vspace{0.5ex}\\
\begin{tabular}{@{}rl}
\rm{(\romannumeral1)} &$\Phi(u_n)\to\Phi(u)$ in $\mathcal{D}^{1,2}(\R^3)$;\vspace{0.5ex}\\
\rm{(\romannumeral2)} &$\int_{ \R^3}\phi_{u_n}u_n^2 dx\to\int_{ \R^3}\phi_{u}u^2 dx$;\vspace{0.5ex}\\
\rm{(\romannumeral3)} &$\int_{ {\R^3}} \phi_{u_n}u_n\varphi dx\to\int_{ \R^3}\phi_{u}u\varphi dx$,\hspace{1ex}  $\forall \varphi \in H^{1}( \R^3)$.
\end{tabular}	
\end{lemma}	
\begin{proof}
We borrow an idea from Cerami et al.\cite{Cerami2} to prove this lemma.
\par
(i)
By definition of $\Phi$ and $L_u$, we have 
$$||\Phi[u]||_{\mathcal{D}^{1,2}} = ||\phi_{u}||_{\mathcal{D}^{1,2}}=||L_u||_{\mathcal{L}(\mathcal{D}^{1,2},\R)},$$
to prove (i), it is enough to show that
$$||L_{u_n}-L_u||_{\mathcal{L}(\mathcal{D}^{1,2},\R)}\to 0, \quad n\to \infty.$$
For all $v\in \mathcal{D}^{1,2}(\R^3)$, then $v\in L^{6}(\R^3)$, thus for any $\sigma  >0$, there exists $R>0$ large enough, such that $|v|_{6, B_{R}^C(0)}<\sigma$.  Thus, for all $v\in \mathcal{D}^{1,2}(\R^3)$, we obtain
\begin{equation}\label{L221}
\begin{aligned}
0\le|L_{u_n}(v)-L_u (v)|&=\left|\int_{\R^3}(u_n^2 -u^2)vdx\right|\\
&\le \int_{B_{R}^C(0)}|u_n^2 -u^2||v|dx+\int_{B_{R}(0)}|u_n^2 -u^2||v|dx\\
&\le|v|_{6, B_{R}^C(0)}|u_n -u|_{\frac{12}{5}}|u_n+u|_{\frac{12}{5}}+\int_{B_{R}(0)}|u_n^2 -u^2||v|dx\\
&\le C\sigma +\int_{B_{R}(0)}|u_n^2 -u^2||v|dx.
\end{aligned}
\end{equation}
Furthermore, we may assume, going if necessary to a subsequence, $u_n\to u$ in $L_{loc}^{\frac{12}{5}}(\R^3)$. Hence, we have
\begin{equation}\label{L222}
\int_{B_{R}(0)}|u_n^2 -u^2||v|dx\le\left( \int_{{B_R}(0)}|u_n -u|^{\frac{12}{5}}dx\right)^{\frac{5}{12}}\left( \int_{B_{R}(0)}|u_n +u|^{\frac{12}{5}}dx\right)^{\frac{5}{12}} |v|_6 =o_n(1).
\end{equation}
Combining \eqref{L221} and \eqref{L222}, the desired conclusion is obtained.
\par
(ii)
At first, replacing $v$ by $\phi_{u_n}$ and repeating the argument used in the prove of (i), we can obtain that
$\int_{ \R^3}\phi_{u_n}(u_n^2 -u^2) dx\to 0$, as $n\to\infty$ holds. On the other hand, we observe that
$$\int_{ \R^3}(\phi_{u_n}-\phi_{u})u^2dx\le|u|_{\frac{12}{5}}^2|\phi_{u_n}-\phi_{u}|_6=o(1).$$
Indeed, since the embedding $\mathcal{D}^{1,2}(\R^3)\hookrightarrow L^6(\R^3)$ is continuous and $\Phi(u_n)\to\Phi(u)$ in $\mathcal{D}^{1,2}(\R^3)$.
Hence, we obtain   
\begin{equation*}
\left|\int_{ \R^3}(\phi_{u_n}u_n^2 -\phi_{u}u^2) dx\right|\le\left|\int_{ \R^3}\phi_{u_n}(u_n^2 -u^2) dx\right|+\left|\int_{ \R^3}(\phi_{u_n}-\phi_{u})u^2dx\right|=o_n(1),
\end{equation*} 
as desired.
\par
(iii)
 To show (iii), we first prove that 
 \begin{equation}\label{26}
 \left|\int_{ \R^3}\phi_{u}(u_n -u)\varphi dx\right|=o_n(1).
 \end{equation}
 In fact, using $\varphi\in H^1(\R^3)$ and $\phi_{u}\in\mathcal{D}^{1,2}(\R^3)$, it follows from H\"{o}lder inequality that $\phi_{u}\varphi\in L^2(\R^3)$.
 Since $u_n\rightharpoonup u $ in $H^{1}(\R^3) $and, then, in $L^2(\R^3) $. Thus, it is easy to see that \eqref{26} holds. On the other hand, we have
 \begin{equation}\label{27}
 \left|\int_{ \R^3}(\phi_{u_n}-\phi_{u})u_n \varphi dx\right|\le |u_n|_2|\phi_{u_n}-\phi_{u}|_6|\varphi|_3=o_n(1).
 \end{equation}
 Combining \eqref{26} and \eqref{27}, we can get that
 \begin{equation*}
 \begin{aligned}
 \left|\int_{ \R^3}(\phi_{u_n}u_n\varphi-\phi_{u}u\varphi)dx\right|
 &\le\left|\int_{ \R^3}\phi_{u}(u_n -u)\varphi dx\right|+\left|\int_{ \R^3}(\phi_{u_n}-\phi_{u})u_n \varphi dx\right|\\
 &=o_n(1),
 \end{aligned}
 \end{equation*}
 as desired. 	
\end{proof}
\begin{remark}
From Lemma \ref{L21} \text{(iv)}  and \cite[Lemma 3.1]{Cerami}, we obtain that $F\in C^{2}(H^1(\R^3),\R) $ is well-defined with
\begin{equation*}
F(u)=\int_{\R^3}\phi_{u}u^2 dx.
\end{equation*}
Moreover, the functional $F$ and its derivative $F'$ posses BL-splitting property \cite[Lemma 2.2]{LFZ2}, which is similar to Brezis--Lieb lemma \cite{Wi}.
\end{remark}
\par
Substituting \eqref{sp} into \eqref{11} and making the change of variable $ x\mapsto \varepsilon x$  , we can rewrite \eqref{1} as the following equivalent equation
\begin{equation}\label{22}
-\Delta u+\phi_u u+V(\varepsilon x)u=\sum\limits_{i=1}^{m} Q_i(\varepsilon x)|u|^{q_i-2}u+K(\varepsilon x)|u|^{4}u,\hspace{1ex}x\in\R^3.
\end{equation}
Obviously, $v(x)$ is a solution of \eqref{1} if and only if $u(x)=v(\varepsilon x)$ is a solution of  \eqref{22}. For any $\varepsilon >0$, we define the Hilbert space $H_{\varepsilon} =\{u\in H^{1}(\R^3):\int_{\R^3}V(\varepsilon x)|u|^2 dx<\infty\}$ 
equipped with the norm
\begin{equation*}
||u||_\varepsilon ^2 = \int_{\R^3}(|\nabla u|^2 +V(\varepsilon x)|u|^2) dx .
\end{equation*}
Since $V(x)$ is positive and bounded for all $x\in\R^3$, we have $H_\varepsilon=H^1(\R^3)$ and the norm $||\cdot||_{\varepsilon}$ is equivalent to $||\cdot||$. 
At this step, under our assumptions it is standard to see that \eqref{22} is variational and its solutions are the critical points of 
the functional $I_\varepsilon :H_{\varepsilon}\to \R$ given by
\begin{equation}
I_\varepsilon(u)=\frac{1}{2}||u||^2_\varepsilon +\frac{1}{4}\int_{ \R^3}\phi_{u}u^2 dx - \sum\limits_{i=1}^{m} \frac{1}{q_i}\int_{ \R^3}Q_i (\varepsilon x)|u|^{q_i}dx-\frac{1}{6}\int_{ \R^3} K(\varepsilon x)|u|^{6}dx.
\end{equation}
Moreover, $I_{\varepsilon} \in C^{2}(H_{\varepsilon},\R)$.
Next, we introduce the Nehari manifold  associated to $I_\varepsilon$ by
\begin{equation*}
 \mathcal{N}_\varepsilon= \{u\in H_\varepsilon \backslash \{0\}:\langle I_{\varepsilon}'(u),u \rangle=0\}.
\end{equation*}
Clearly, $u\in \mathcal{N}_\varepsilon$ if and only if 
\begin{equation}
||u||^2_\varepsilon +\int_{ \R^3}\phi_{u}u^2 dx = \sum\limits_{i=1}^{m} \int_{ \R^3}Q_i (\varepsilon x)|u|^{q_i}dx + 
\int_{ \R^3} K(\varepsilon x)|u|^{6}dx.
\end{equation}
Moreover, $I_\varepsilon$ is bounded from below on $\mathcal{N}_\varepsilon$. So we can consider the following problem:
\begin{equation}\label{lstene}
c_{\varepsilon}:=\inf\limits_{u\in \mathcal{N}_\varepsilon} I_{\varepsilon}(u).
\end{equation}
Now, we summarize some properties of $I_\varepsilon$ on $\mathcal{N}_\varepsilon$.
\begin{lemma}\label{L23}
For any $u\in H_\varepsilon \backslash \{0\}$, the following statements hold true.\vspace{0.5ex}\\
\begin{tabular}{@{}rl}
\rm{(\romannumeral1)} &There exists a unique $t_{\varepsilon}=t_{\varepsilon}(u)>0$ such that $t_{\varepsilon}u\in \mathcal{N}_\varepsilon$ and\\
\rule[-10pt]{0pt}{25pt}&\multicolumn{1}{c}{$I_{\varepsilon}(t_{\varepsilon}u):=\max\limits_{t\geq0} I_{\varepsilon}(tu)$.}\\
\rm{(\romannumeral2)} &There exist constants $0<\alpha_1 <\alpha_2$ independent of $\varepsilon >0$, such that $\alpha_1\le t_{\varepsilon}\le \alpha_2$.\\ 
\rm{(\romannumeral3)} &$I_\varepsilon$ is coercive  and bounded from below on $\mathcal{N}_\varepsilon$. \\
\rm{(\romannumeral4)} &There exists $\kappa  >0$ independent of $\varepsilon >0$, such that \\
\rule[-8pt]{0pt}{25pt}&\multicolumn{1}{c}{$||u||_\varepsilon\geq \kappa $,\quad $I_\varepsilon(u)\geq\frac{q_{i_0}-2}{2q_{i_0}}\kappa^2$, \quad $\forall u_{\varepsilon}\in \mathcal{N}_\varepsilon$.}\\
\end{tabular}
\end{lemma}
\begin{proof}
(i) For $t>0$, we set
\begin{equation*}
\begin{aligned}
h(t):=I_\varepsilon(tu)=\frac{t^2}{2}||u||^2_\varepsilon +\frac{t^4}{4}\int_{ \R^3}\phi_{u}u^2 dx - \sum\limits_{i=1}^{m} \frac{t^{q_i}}{q_i}\int_{ \R^3}Q_i (\varepsilon x)|u|^{q_i}dx-\frac{t^6}{6}\int_{ \R^3} K(\varepsilon x)|u|^{6}dx.
\end{aligned}
\end{equation*}
From the Sobolev embedding inequalities \eqref{spine}, we obtain 
$$h(t)\geq \frac{t^2}{2}||u||^2_\varepsilon - \sum\limits_{i=1}^{m} \frac{t^{q_i}}{q_i}C_i||u||^{q_i}_\varepsilon-\frac{t^6}{6}C_{m+1}||u||^6_\varepsilon.$$
Since $4<q_i<6$, it is easy to see that $h(t)>0$ for small $t>0$.
Moreover,  it follows from Lemma \ref{L21} (iv) that, 
$$h(t)\le\frac{t^2}{2}||u||^2_\varepsilon +\frac{t^4}{4}C||u||^4_\varepsilon +\sum\limits_{i=1}^{m}\frac{t^{q_i}}{q_i}C_i||u||^{q_i}_\varepsilon-\frac{t^6}{6}\int_{ \R^3} K(\varepsilon x)|u|^{6}dx \to-\infty,$$
as $t\to\infty$. Consequently, $\max\limits_{t\geq 0} h(t)$ is achieved at $t_{\varepsilon}=t_{\varepsilon}(u)>0$, hence $h'(t_\varepsilon)=0$ and $t_{\varepsilon}u\in \mathcal{N}_\varepsilon$.
\par
Next, we show the uniqueness of $t_\varepsilon$. Arguing indirectly, if there exist $0<t_1<t_2$ such that $t_1 u, t_2 u\in \mathcal{N}_\varepsilon$. Let,
$$
f(t):=-{t^2}\int_{ \R^3}\phi_{u}u^2 dx + \sum\limits_{i=1}^{m} {t^{q_i-2}}\int_{ \R^3}Q_i (\varepsilon x)|u|^{q_i}dx+{t^4}\int_{ \R^3} K(\varepsilon x)|u|^{6}dx.
$$
Taking account of $4<q_1<q_2<\cdots<q_m<6$ and $(f_2)$, we have $f(t)$ is a strictly increasing function on any interval where $f(t)>0$.
Then, we deduce from $h'(t_1)=0$ and $h'(t_2)=0$ that
$$f(t_1)=||u||_\varepsilon^2\hspace{1ex}\text{and}\hspace{1ex} f(t_2)=||u||_\varepsilon^2,$$
which is a contradiction.	
\par
(ii) Since $t_\varepsilon u\in \mathcal{N}_\varepsilon$, then $f(t_\varepsilon)=0$, i.e.,
\begin{equation}
-{t_\varepsilon^2}\int_{ \R^3}\phi_{u}u^2 dx + \sum\limits_{i=1}^{m} {t_\varepsilon^{q_i-2}}\int_{ \R^3}Q_i (\varepsilon x)|u|^{q_i}dx+{t_\varepsilon^4}\int_{ \R^3} K(\varepsilon x)|u|^{6}dx=||u||_\varepsilon^2.
\end{equation}
It follows from that 
\begin{equation}\label{213}
||u||_\varepsilon^2\le \sum\limits_{i=1}^{m} t_\varepsilon^{q_i-2}C_i||u||^{q_i}_\varepsilon+t_\varepsilon^4C_{m+1}||u||^6_\varepsilon,
\end{equation}
and
$${t_\varepsilon^4}\int_{ \R^3} |u|^{6}dx\le ||u||_\varepsilon^2+ {t_\varepsilon^2}C||u||^4_\varepsilon +\sum\limits_{i=1}^{m}t^{q_i-2}C_i||u||^{q_i}_\varepsilon. $$
Then, it is easy to see that there exist constants $0<\alpha_1 <\alpha_2$ independent of $\varepsilon >0$, such that $\alpha_1\le t_{\varepsilon}\le \alpha_2$.
\par 
(iii) For $u\in\mathcal{N}_\varepsilon$, we have 
\begin{equation}\label{bddlow}
\begin{aligned}
I_\varepsilon (u)
= & \hspace{0.4em}I_\varepsilon (u)-\frac{1}{q_{i_0}}\langle I_{\varepsilon}'(u),u \rangle\\
=& \left(\frac{1}{2}-\frac{1}{q_{i_0}}\right)||u||^2_\varepsilon +\left(\frac{1}{4}-\frac{1}{q_{i_0}}\right)\int_{ \R^3}\phi_{u}u^2 dx \\
&- \sum\limits_{i=1}^{i_0-1} \left(\frac{1}{q_i}-\frac{1}{q_{i_0}}\right)\int_{ \R^3}Q_i (\varepsilon x)|u|^{q_i}dx+\sum\limits_{i=i_0+1}^{m}\left (\frac{1}{q_{i_0}}-\frac{1}{q_i}\right)\int_{ \R^3}Q_i (\varepsilon x)|u|^{q_i}dx\\
&+\left({\frac{1}{q_{i_0}}-\frac{1}{6}}\right)\int_{ \R^3} K(\varepsilon x)|u|^{6}dx.\\
\geq&\left(\frac{1}{2}-\frac{1}{q_{i_0}}\right)||u||^2_\varepsilon>0,
\end{aligned}
\end{equation}
which implies that $I_\varepsilon$ is coercive  and bounded from below on $\mathcal{N}_\varepsilon$.
\par 
(iv) The conclusion is immediate by (iii) and taking $t_\varepsilon=1$ in \eqref{213}.
\end{proof}
\begin{lemma}\label{L24}
$I_\varepsilon$ has the mountain pass geometry structure.\\
\begin{tabular}{@{}rl}
\rm{(\romannumeral1)} &There exist $\alpha, \rho>0$ independent of $\varepsilon$, such that $I_\varepsilon(u)\geq \alpha$ for $||u||_\varepsilon = \rho $;\\
\rm{(\romannumeral2)} &There exists an $e\in H_{\varepsilon}$ satisfying $||e||_\varepsilon \geq\rho $ such that $I_{\varepsilon}(e)<0$.
\end{tabular}
\end{lemma}
\begin{proof}
(i) For any $u\in H_\varepsilon \backslash\{0\}$, we deduce from the Sobolev embedding inequalities \eqref{spine} that
\begin{equation*}
\begin{aligned}
I_\varepsilon(u)&=\frac{1}{2}||u||^2_\varepsilon +\frac{1}{4}\int_{ \R^3}\phi_{u}u^2 dx - \sum\limits_{i=1}^{m}\frac{1}{q_i}\int_{ \R^3}Q_i (\varepsilon x)|u|^{q_i}dx-\frac{1}{6}\int_{ \R^3} K(\varepsilon x)|u|^{6}dx\\
&\geq \frac{1}{2}||u||^2_\varepsilon - \sum\limits_{i=1}^{m} \frac{1}{q_i}C_i||u||^{q_i}_\varepsilon-\frac{1}{6}C_{n+1}||u||^6_\varepsilon.
\end{aligned}
\end{equation*}
Set $||u||_\varepsilon=\rho$ small enough, such that  $I_\varepsilon(u)\geq \alpha$.
\par
(ii) For any $u\in H_\varepsilon \backslash\{0\}$.
\begin{equation*}
\begin{aligned}
I_\varepsilon(tu)&=\frac{t^2}{2}||u||^2_\varepsilon +\frac{t^4}{4}\int_{ \R^3}\phi_{u}u^2 dx - \sum\limits_{i=1}^{m} \frac{t^{q_i}}{q_i}\int_{ \R^3}Q_i (\varepsilon x)|u|^{q_i}dx-\frac{t^6}{6}\int_{ \R^3} K(\varepsilon x)|u|^{6}dx\\
&\le\frac{t^2}{2}||u||^2_\varepsilon +\frac{t^4}{4}C||u||^4_\varepsilon -\sum\limits_{i=1}^{m}\frac{t^{q_i}}{q_i}C_i||u||^{q_i}_\varepsilon-\frac{t^6}{6}\int_{ \R^3} K(\varepsilon x)|u|^{6}dx.
\end{aligned}
\end{equation*}
Since $4<q_i<6$ and $K(x)\geq 0$, there exist $e:=t'u$ for some $t'>0$ large enough such that  $||e||_\varepsilon \geq\rho $ and $I_{\varepsilon}(e)<0$.
\end{proof}
\begin{lemma}\label{L25}
For any $\varepsilon>0$, we can define 
\begin{equation*}
c_{\varepsilon}^{*}:=\inf\limits_{u\in H_\varepsilon \backslash \{0\}}\max\limits_{t\geq 0} I_{\varepsilon}(tu),\quad
c_{\varepsilon}^{**}:=\inf\limits_{\gamma \in \Gamma }\sup\limits_{t\in[0,1]} I_{\varepsilon}(\gamma(t)),
\end{equation*}
where $\Gamma=\{\gamma\in C([0,1],H_\varepsilon);\gamma(0)=0,I_\varepsilon(\gamma(1))<0)\}.$\\
Then, 
\begin{equation}\label{3eq}
c_{\varepsilon}=c_{\varepsilon}^{*}=c_{\varepsilon}^{**}.
\end{equation}
\end{lemma}
\begin{proof}
By Lemma \ref{L23} (i), we have 
$$c_{\varepsilon}=\inf\limits_{u\in \mathcal{N}_\varepsilon}I_\varepsilon(u)=\inf\limits_{u\in H_\varepsilon \backslash \{0\}}I_\varepsilon(t_\varepsilon u)=\inf\limits_{u\in H_\varepsilon \backslash \{0\}}\max\limits_{t\geq 0}I_\varepsilon(tu)=c_{\varepsilon}^{*}.$$
Moreover, by Lemma \ref{L24} (ii), for any $u\in H_\varepsilon \backslash \{0\}$, there exists $k>0$ large enough, such that $I_\varepsilon(ku)<0$, set $\beta(t)=tku$, $t\in[0,1]$, then $\beta\in\Gamma$. Indeed, $\beta(0)=0$, $\beta(1)=ku$ and $I_\varepsilon(\beta(1))<0$. Thus,
$$\max\limits_{t\geq 0} I_{\varepsilon}(tu)=\sup\limits_{t\in[0,1]} I_{\varepsilon}(tku)=\sup\limits_{t\in[0,1]} I_{\varepsilon}(\beta(t))\geq c_{\varepsilon}^{**}.$$
It follows that $c_{\varepsilon}^{*}\geq c_{\varepsilon}^{**}$. On the other hand, the Nehari manifold $\mathcal{N}_\varepsilon$ separates $H_\varepsilon$ into two components
$$H_\varepsilon^+ =\{u\in H_\varepsilon:\langle I_{\varepsilon}'(u),u \rangle >0\}\cup\{0\}$$
and
$$H_\varepsilon^- =\{u\in H_\varepsilon:\langle I_{\varepsilon}'(u),u \rangle <0\}.$$
It follows from \eqref{bddlow} that $I_\varepsilon(u)\geq 0$ for $u\in H_\varepsilon^+$, and  $\frac{1}{q_{i_0}}\langle I_{\varepsilon}'(\gamma(1)),\gamma(1)\rangle\leq I_\varepsilon(\gamma(1))<0$. Thus any $\gamma\in\Gamma$ has to cross $\mathcal{N}_\varepsilon$, since $ \gamma(0)\in H_\varepsilon^+ $ and $\gamma(1)\in H_\varepsilon^- $, and so $c_{\varepsilon}^{**}\geq c_{\varepsilon}$. The proof is complete.
\end{proof}
In order to study \eqref{22}, we need some results about the autonomous problem of \eqref{22}. For $a>0$, $b_{i_0}>0$, $(j-i_0)b_j>0(j=1,2,\cdots,m+1; j\neq i_0)$, consider the autonomous problem:
\begin{equation}\label{ap}
-\Delta u+\phi_u u+au=\sum\limits_{i=1}^{m} b_i|u|^{q_i-2}u+b_{m+1}|u|^{4}u,\quad x\in\R^3.
\end{equation}	
The associated energy functional is 
\begin{equation*}
\begin{aligned}
I_{ab}(u)=\frac{1}{2}||u||^2_{a} + \frac{1}{4}\int_{ \R^3}\phi_{u}u^2 dx - \sum\limits_{i=1}^{m} \frac{b_i}{q_i}\int_{ \R^3}|u|^{q_i}dx-\frac{b_{m+1}}{6}\int_{ \R^3}|u|^{6}dx,
\end{aligned}
\end{equation*}
where $||u||^2_{a}=\int_{\R^3}(|\nabla u|^2 +a|u|^2) dx$.
By Lemma \ref{L25}, we have 
\begin{equation}\label{217}
c_{ab}:=\inf\limits_{u\in \mathcal{N}_{ab}} I_{ab}(u)=\inf\limits_{u\in H_\varepsilon \backslash \{0\}}\max\limits_{t\geq 0} I_{ab}(tu)=\inf\limits_{\gamma \in \Gamma_{ab} }\sup\limits_{t\in[0,1]} I_{ab}(\gamma(t)),
\end{equation}
where $\mathcal{N}_{ab}= \{u\in H_\varepsilon \backslash \{0\}:\langle I_{ab}'(u),u \rangle=0\}$. $\Gamma=\{\gamma\in C([0,1],H_\varepsilon);\gamma(0)=0,I_\varepsilon(\gamma(1)<0)\}.$
\begin{lemma}\label{L27}
	Problem \eqref{ap} has at least a positive ground state solution in $H_\varepsilon$.
\end{lemma}	
\begin{proof}
The idea of the proofs is sketched as follows. At first,  similar  to Lemma \ref{L24}, we can show that $I_{ab}$ satisfies the mountain pass geometry structure. Moreover, similar to \eqref{bddlow}, we can show the (PS) sequence is bounded in $H_\varepsilon$. Secondly, with the help of technique of Brezis and Nirenberg \cite{{Bre}}, we can obtain the estimation of the mountain pass critical level under our conditions for \eqref{ap}, which can recover the compactness of the (PS) sequence. Finally, it follows from the Lions' concentration-compactness principle   \cite{PL1,PL2} and some standard arguments that \eqref{ap} has a positive ground state solution. For the details of the proof, we refer readers to \cite{hezou,HLR2}.
\end{proof}
The following lemma describes a comparison between the Mountain-Pass values for different parameters, which will play a very important role in obtaining the existence results. 
\begin{lemma}\label{L277}
	For $a,\tilde{a} >0$, $b_{i_0}, \tilde{b}_{i_0}>0$, $(j-i_0)b_j>0$, $(j-i_0)\tilde{b}_j>0$$(j=1,2,\cdots,m+1)$. If
	$$min\{a-\tilde{a}, \tilde{b}_j-b_j\}\geq0,$$
	then $c_{ab}\geq c_{\tilde{a}\tilde{b}}$. In particular, if $max\{a-\tilde{a}, \tilde{b}_j-b_j\}>0$ also holds, then $c_{ab} > c_{\tilde{a}\tilde{b}}$.
\end{lemma}	
\begin{proof}
From Lemma \ref{L27}, we choose $u_{ab}$ be a solution of problem \eqref{ap} such that $c_{ab}=I_{ab}(u_{ab})$. There holds 
$$c_{ab}=\max\limits_{t\geq 0}I_{ab}(tu_{ab}).$$
By the similar arguments to Lemma \ref{L23} (i), there exists $t_1>0$ such that $t_1u_{ab}\in \mathcal{N}_{\tilde{a}\tilde{b}}$. let $\tilde{u}_{\tilde{a}\tilde{b}}=t_1u_{ab}$ be such that $I_{\tilde{a}\tilde{b}}(\tilde{u}_{\tilde{a}\tilde{b}})=\max\limits_{t\geq 0}I_{\tilde{a}\tilde{b}}(tu_{ab})$. Then, we see that
$$\begin{aligned}
c_{ab}=I_{ab}(u_{ab})\geq& I_{ab}(\tilde{u}_{\tilde{a}\tilde{b}})\\
=&I_{\tilde{a}\tilde{b}}(\tilde{u}_{\tilde{a}\tilde{b}})+\frac{a-\tilde{a}}{2}||u||_{\mathcal{D}^{1,2}}\\
 &- \sum\limits_{i=1}^{m} \frac{b_i-\tilde{b}_i}{q_i}\int_{ \R^3}|u|^{q_i}dx-\frac{b_{m+1}-\tilde{b}_{m+1}}{6}\int_{ \R^3}|u|^{6}dx\\
 \geq &c_{\tilde{a}\tilde{b}}.
\end{aligned}$$
In particular, if $max\{a-\tilde{a}, \tilde{b}_j-b_j\}>0$ also holds, then the above inequality implies that  $c_{ab} > c_{\tilde{a}\tilde{b}}$. The proof is complete.
\end{proof}
 \vskip2mm
 {\section{The existence of ground state solution }\label{scr}\label{sec3}}
\setcounter{equation}{0}
\vskip2mm
In this section, we will establish compactness lemma for $I_\varepsilon$, and prove the existence of ground state solution to \eqref{22}. Firstly, we need to give the energy functional  associated with \eqref{limpro} by
\begin{equation*}
I_\infty(u)=\frac{1}{2}||u||^2_{\mathcal{D}^{1,2}}+\frac{1}{2}\int_{ \R^3}V_{\infty}u^2 dx +\frac{1}{4}\int_{ \R^3}\phi_{u}u^2 dx - \sum\limits_{i=1}^{m} \frac{1}{q_i}\int_{ \R^3}Q_i^{\infty}|u|^{q_i}dx-\frac{1}{6}\int_{ \R^3} K^{\infty}|u|^{6}dx.
\end{equation*}
Considering the following minimization problem:
\begin{equation*}
 c_{\infty}:=\inf\limits_{u\in \mathcal{N}_\infty} I_{\infty}(u),
\end{equation*}
where $\mathcal{N}_\infty= \{u\in H_\varepsilon \backslash \{0\}:\langle I_{\infty}'(u),u \rangle=0\}.$
Moreover, we denote the energy functional for \eqref{gs} by
\begin{equation}
\begin{aligned}
I^s(u)=&\frac{1}{2}||u||^2_{\mathcal{D}^{1,2}}+\frac{1}{2}\int_{ \R^3}V(s)u^2 dx +\frac{1}{4}\int_{ \R^3}\phi_{u}u^2 dx\\
&- \sum\limits_{i=1}^{m} \frac{1}{q_i}\int_{ \R^3}Q_i(s)|u|^{q_i}dx-\frac{1}{6}\int_{ \R^3} K(s)|u|^{6}dx.
\end{aligned}
\end{equation}
Finally, we 
define the ground energy function 
$$G(s):=\inf\limits_{u\in \mathcal{N}^s} I^s(u),$$
where $\mathcal{N}^s= \{u\in H_\varepsilon \backslash \{0\}:\langle (I^s)'(u),u \rangle=0\}$, which was mentioned in Section \ref{sec1} and will play an important role in estimating the critical energy level.

\begin{lemma}\label{Lc}
For $s\in \R^3$, the ground energy function $G(s)$ is locally Lipschitz continuous.
\end{lemma}
The proof of Lemma \ref{Lc} is essentially similar to that in \cite[Lemma 2.3, Lemma 2.4 ]{XB}, so we omit it here.
\begin{lemma}\label{vnsh}(\cite[Lemma 1.21]{Wi})
	Let $r>0$ and $\{u_n\}$ is bounded in $H^{1}(\R^N)$. If
	$$\sup\limits_{y\in\R^N}\int_{B_r (y)}|u_n|^2\to 0,\quad n\to \infty,$$
	then $u_n \to 0$ in $L^s(\R^N) $ for $2<s<2^*$.
\end{lemma}
To estimate the critical energy level for the critical Schr\"{o}dinger--Poisson system involving more than four potentials,  we have an important upper bound for the least energy $c_\varepsilon$ defined in \eqref{lstene} via the definition of the ground energy function $G(s)$.
\begin{lemma}\label{cric1}
There exists $C>0$ independent of $\varepsilon$, such that $c_\varepsilon\geq C$. Furthermore, 
\begin{equation}\label{32}
\limsup\limits_{\varepsilon\to 0}c_\varepsilon\le c_0:=\inf\limits_{s\in \R^3}G(s).
\end{equation}
\end{lemma}
\begin{proof}
Taking $a=\inf\limits_{x\in\R^{3}}V(x)$, $b_i=\sup\limits_{x\in\R^{3}}Q_i(x)(1\le i\le m)$, and  $b_m+1=\sup\limits_{x\in\R^{3}}K(x)$ in \eqref{217} and using Lemma \ref{L277}, we can obtain that $c_\varepsilon\geq c_{ab}>0$. We only need to show \eqref{32}.
It follows from $c_0<c_\infty$ and Lemma \ref{Lc} that there exists $s_0\in\R^3$ such that $G(s_0)=c_0$. Furthermore, by Lemma \ref{L27}, there exists $u_{s_0}\in \mathcal{N}^{s_0}$, i.e.,
\begin{equation}\label{tto1}
\begin{aligned}
&||u_{s_0}||^2_{\mathcal{D}^{1,2}}+ 
\int_{ \R^3} V(s_0)|u_{s_0}|^{2}dx+\int_{ \R^3}\phi_{u_{s_0}}u_{s_0}^2 dx\\
 = &\sum\limits_{i=1}^{m} \int_{ \R^3}Q_i (s_0)|u_{s_0}|^{q_i}dx + 
\int_{ \R^3} K(s_0)|u_{s_0}|^{6}dx,
\end{aligned}
\end{equation}
such that $I^{s_0}(u_{s_0})=c_0$. Set $\omega_\varepsilon(x)=u_{s_0}\left(x-\frac{s_0}{\varepsilon}\right)$,  from Lemma \ref{L23}(i) and (ii), we know that there exists a unique bounded $t_\varepsilon>0$ such that $t_\varepsilon\omega_\varepsilon\in\mathcal{N}_\varepsilon$, i.e.,
$$t_\varepsilon||\omega_\varepsilon||^2_\varepsilon +t_\varepsilon^3\int_{ \R^3}\phi_{\omega_\varepsilon}\omega_\varepsilon^2 dx = \sum\limits_{i=1}^{m}t_\varepsilon^{q_i-1}  \int_{ \R^3}Q_i (\varepsilon x)|\omega_\varepsilon|^{q_i}dx + t_\varepsilon^6
\int_{ \R^3} K(\varepsilon x)|\omega_\varepsilon|^{6}dx,$$
and so
\begin{equation}
\begin{aligned}
&t_\varepsilon||u_{s_0}||^2_{\mathcal{D}^{1,2}} +t_\varepsilon\int_{ \R^3}V(\varepsilon x+s_0)|u_{s_0}|^{2}dx+t_\varepsilon^3\int_{ \R^3}\phi_{u_{s_0}}u_{s_0}^2 dx \\
= &\sum\limits_{i=1}^{m} t_\varepsilon^{q_i-1}\int_{ \R^3} Q_i (\varepsilon x+s_0)|u_{s_0}|^{q_i}dx + t_\varepsilon^6
\int_{ \R^3} K(\varepsilon x+s_0)|u_{s_0}|^{6}dx.
\end{aligned}
\end{equation}
\par
Since $V(x)$ is bounded, then there exists a constant $M$ such that $|V(\varepsilon x+s_0)-V(s_0)|\le M$. From \eqref{spine}, we have $u_{s_0}\in L^p(\R^3)(2\le p\le6)$, then for any $\sigma>0$, there exists $R>0$ such that $|u_{s_0}|_{2,B_R^C(0)}<\frac{\sigma}{2M}$. Thus, we have
\begin{equation*}
 \int_{\R^3\backslash B_{R}(0)}(V(\varepsilon x+s_0)-V(s_0)|u_{s_0}|^2dx<\frac{\sigma}{2},
\end{equation*}
and obviously we observe that
\begin{equation*}
\int_{B_{R}(0)}(V(\varepsilon x+s_0)-V(s_0)|u_{s_0}|^2dx<\frac{\sigma}{2},
\end{equation*}
as $\varepsilon\to 0$.
Which implies that
 \begin{equation}
 \int_{ \R^3}V(\varepsilon x+s_0)|u_{s_0}|^2dx\to\int_{ \R^3}V(s_0)|u_{s_0}|^2dx,
 \end{equation}
 as $\varepsilon\to 0$. Similarly, we can obtain that
  \begin{equation}
 \int_{ \R^3}K(\varepsilon x+s_0)|u_{s_0}|^6dx\to\int_{ \R^3}K(s_0)|u_{s_0}|^6dx,
 \end{equation}
 and
  \begin{equation}\label{to1}
 \int_{ \R^3}Q_i(\varepsilon x+s_0)|u_{s_0}|^2dx\to\int_{ \R^3}Q_i(s_0)|u_{s_0})|^2dx,(1\le i \le m),
 \end{equation}
  as $\varepsilon\to 0$. From \eqref{tto1}--\eqref{to1}, we have $t_\varepsilon\to 1$, as $\varepsilon\to 0$.
Now observe that
  \begin{equation}
 \begin{aligned}
 c_\varepsilon=&\inf\limits_{u\in \mathcal{N}_\varepsilon}I_\varepsilon(u)\\
 \le& I_\varepsilon(t_\varepsilon \omega_\varepsilon)\\
 =&I^{s_0}(t_\varepsilon \omega_\varepsilon)+\frac{t_\varepsilon^2}{2}\int_{\R^3}(V(\varepsilon x)-V(s_0))| \omega_\varepsilon|^2dx\\
 &-\sum\limits_{i=1}^{m}\frac{t_\varepsilon^{q_i}}{q_i} \int_{ \R^3} (Q_i (\varepsilon x)-Q_i(s_0))|\omega_\varepsilon|^{q_i}dx -\frac{t_\varepsilon^6}{6} 
 \int_{ \R^3} (K(\varepsilon x)-K(s_0)|\omega_\varepsilon|^{6}dx\\
 =&I^{s_0}(t_\varepsilon u_{s_0})+\frac{t_\varepsilon^2}{2}\int_{\R^3}(V(\varepsilon x+s_0)-V(s_0))| u_{s_0}|^2dx\\
 &-\sum\limits_{i=1}^{m}\frac{t_\varepsilon^{q_i}}{q_i} \int_{ \R^3} (Q_i (\varepsilon x+s_0)-Q_i(s_0))|u_{s_0}|^{q_i}dx\\ &-\frac{t_\varepsilon^6}{6} 
 \int_{ \R^3} (K(\varepsilon x+s_0)-K(s_0)|u_{s_0}|^{6}dx.
 \end{aligned}
  \end{equation}
The property of $t_\varepsilon$ discussed above implies that  $I^{s_0}(t_\varepsilon u_{s_0})\to I^{s_0}( u_{s_0})=c_0$ as $\varepsilon\to 0$. Now the desired conclusion is obtained. 
\end{proof}
For any $\sigma>0$, we consider the function $u_{\sigma,x_0}$ defined by 
$$u_{\sigma,x_0}(x)=\frac{(3\sigma^2)^\frac{1}{4}}{(\sigma^2+|x-x_0|^2)^{\frac{1}{2}}},$$
which is a solution of the critical problem $-\Delta u=u^5$ in $\R^3$(see \cite{Wi,ms}), and $x_0$ is given in condition $(f_4)$. Let $U_{\sigma,x_0}=\xi(x-x_0)u_{\sigma,x_0}$, where $\xi(x)\in C^{\infty}_0(B_{2R}(0),[0,1])$ satisfies $\xi(x)\equiv 1$ on $B_R(0)$, where $R$ is a positive constant. As in \cite{Bre,ms}, the following estimations hold: 
\begin{equation}
\int_{ \R^3} |\nabla U_{\sigma,x_0}|^2dx=S^{\frac{3}{2}}+O(\sigma)\hspace{1ex}\text{and} \hspace{1ex}\int_{ \R^3} |U_{\sigma,x_0}|^6dx=S^{\frac{3}{2}}+O(\sigma^3),
\end{equation}
and for any $t\in[2,6)$,
\begin{equation}\label{cutest}
|U_{\sigma,x_0}|^t_t=\begin{cases}
O(\sigma^{\frac{t}{2}}),& t\in[2,3),\\
O(\sigma^{\frac{3}{2}}|ln\sigma|),& t=3, \\
O(\sigma^{\frac{6-t}{2}}),&t\in(3,6).
\end{cases}
\end{equation}
\par
In the following lemma, we compare the minimum level of ground energy function with a suitable number which involves the best Sobolev embedding constant $S$.
\begin{lemma}\label{cric2}
	 $c_{0}=\inf\limits_{s\in\R^3}G(s)<\frac{1}{3}S^{\frac{3}{2}}|K|_{\infty}^{-\frac{1}{2}}$.	\end{lemma}	
\begin{proof}
Using the condition $(f_4)$, there exists $x_0\in\R^3$ such that $K(x_0)=K^\infty$. We consider the following equation 
$$-\Delta u+\phi_u u+V(x_0)u=\sum\limits_{i=1}^{m} Q_i(x_0)|u|^{q_i-2}u+K(x_0)|u|^{4}u,\quad x\in\R^3, $$
with the corresponding energy functional 
\begin{equation}
\begin{aligned}
I^{x_0}(u)=&\frac{1}{2}||u||^2_{\mathcal{D}^{1,2}}+\frac{1}{2}\int_{ \R^3}V(x_0)u^2 dx +\frac{1}{4}\int_{ \R^3}\phi_{u}u^2 dx\\
&- \sum\limits_{i=1}^{m} \frac{1}{q_i}\int_{ \R^3}Q_i(x_0)|u|^{q_i}dx-\frac{1}{6}\int_{ \R^3} K(x_0)|u|^{6}dx.
\end{aligned}
\end{equation}
Moreover, by Lemma \ref{L27}, we can obtain that there exists $u_{x_0}\in \mathcal{N}^{x_0}$ such that $G(x_0)=I^{x_0}(u_{x_0})$. Taking account of Lemma \ref{L23} (i), (ii) and \eqref{3eq}, we have $G(x_0)\le\max\limits_{t\geq 0}I^{x_0}(tU_{\sigma,x_0}),$ and there exists a unique bounded $t_\sigma>0$ such that $I^{x_0}(t_\sigma U_{\sigma,x_0})=\max\limits_{t\geq 0}I^{x_0}(tU_{\sigma,x_0})$.
By the Lemma \ref{L21} and \eqref{cutest}, we can obtain that 
\begin{equation}
\begin{aligned}
I^{x_0}(t_\sigma U_{\sigma,x_0})=&\frac{t_\sigma^2}{2}||U_{\sigma,x_0}||^2_{\mathcal{D}^{1,2}}+\frac{t_\sigma^2}{2}\int_{ \R^3}V(x_0)U_{\sigma,x_0}^2 dx+\frac{t_\sigma^4}{4}\int_{ \R^3}\phi_{ _{U_{\sigma,x_0}}}U_{\sigma,x_0}^2 dx\\ 
&- \sum\limits_{i=1}^{m} \frac{t_\sigma^{q_i}}{q_i}\int_{ \R^3}Q_i(x_0)|U_{\sigma,x_0}|^{q_i}dx-\frac{t_\sigma^6}{6}\int_{ \R^3} K(x_0)|U_{\sigma,x_0}|^{6}dx\\
\le&\frac{t_\sigma^2}{2}\int_{ \R^3}\left(|\nabla U_{\sigma,x_0}|^2+V(x_0)U_{\sigma,x_0}^2\right) dx+\frac{t_\sigma^4}{4S}\left(\int_{ \R^3}|U_{\sigma,x_0}|^{\frac{12}{5}} dx\right)^{\frac{5}{3}}\\ 
&- \sum\limits_{i=1}^{m} \frac{t_\sigma^{q_i}}{q_i}\int_{ \R^3}Q_i(x_0)|U_{\sigma,x_0}|^{q_i}dx-\frac{t_\sigma^6}{6}\int_{ \R^3} K(x_0)|U_{\sigma,x_0}|^{6}dx\\
\le&\sup\limits_{t_\sigma>0}\left\{\frac{t_\sigma^2}{2}\int_{ \R^3}|\nabla U_{\sigma,x_0}|^2dx -\frac{t_\sigma^6}{6}\int_{ \R^3} K(x_0)|U_{\sigma,x_0}|^{6}dx\right\}\\
&- \sum\limits_{i=1}^{m} \frac{t_\sigma^{q_i}}{q_i}\int_{ \R^3}Q_i(x_0)|U_{\sigma,x_0}|^{q_i}dx+CO(\sigma)+CS^{-1}O(\sigma^2).
	\end{aligned}
	\end{equation}
Using the condition $(f_2)$ and $4<q_1<q_2<\cdots<q_m<6$, one has
\begin{equation}
\begin{aligned} I^{x_0}(t_\sigma U_{\sigma,x_0})\le&\frac{1}{3}S^{\frac{3}{2}}|K|_{\infty}^{-\frac{1}{2}}+O(\sigma)+CO(\sigma)+CS^{-1}O(\sigma^2)\\
&+\sum\limits_{i=1}^{i_0}C_iO(\sigma^\frac{6-q_i}{2})-\sum\limits_{i=i_0+1}^{m}C_iO(\sigma^\frac{6-q_i}{2})\\
<&\frac{1}{3}S^{\frac{3}{2}}|K|_{\infty}^{-\frac{1}{2}},
\end{aligned}
\end{equation}
as $\sigma>0$ small enough. Consequently, we deduce from the definition of $c_0$ that 
$$c_0\le G(x_0)\le I^{x_0}(t_\sigma U_{\sigma,x_0})<\frac{1}{3}S^{\frac{3}{2}}|K|_{\infty}^{-\frac{1}{2}}.$$
The desired conclusion is obtained.
\end{proof}
\begin{remark}
By Lemma \ref{cric1} and Lemma \ref{cric2}, we see that $c_\varepsilon<\frac{1}{3}S^{\frac{3}{2}}|K|_{\infty}^{-\frac{1}{2}}$ for $\varepsilon>0$ small enough. Moreover, it follows from Lemma \ref{cric1} and \eqref{14} that $c_\varepsilon<c_\infty$.
\end{remark}

\begin{lemma}\label{psbdd}
If $\{u_n\}\subset H_\varepsilon$ be a $(PS)_c$ sequence  for $I_\varepsilon$, then $\{u_n\}$ is bounded in $H_\varepsilon$.
\end{lemma}
\begin{proof}
Let $\{u_n\}\subset H_\varepsilon$ be a $(PS)_c$ sequence  for $I_\varepsilon$, for n large enough, we have 
\begin{equation}
\begin{aligned}
c+1+||u_n||_\varepsilon
\geq & \hspace{0.4em}I_\varepsilon (u_n)-\frac{1}{q_{i_0}}\langle I_{\varepsilon}'(u_n),u_n \rangle\\
=&\left(\frac{1}{2}-\frac{1}{q_{i_0}}\right)||u_n||^2_\varepsilon +\left(\frac{1}{4}-\frac{1}{q_{i_0}}\right)\int_{ \R^3}\phi_{u_n}u_n^2 dx \\
&- \sum\limits_{i=1}^{i_0-1} \left(\frac{1}{q_i}-\frac{1}{q_{i_0}}\right)\int_{ \R^3}Q_i (\varepsilon x)|u_n|^{q_i}dx\\
&+\sum\limits_{i=i_0+1}^{m}\left (\frac{1}{q_{i_0}}-\frac{1}{q_i}\right)\int_{ \R^3}Q_i (\varepsilon x)|u_n|^{q_i}dx\\
&+\left({\frac{1}{q_{i_0}}-\frac{1}{6}}\right)\int_{ \R^3} K(\varepsilon x)|u_n|^{6}dx\\
\geq&\left(\frac{1}{2}-\frac{1}{q_{i_0}}\right)||u_n||^2_\varepsilon.
\end{aligned}
\end{equation}
It follows that $\{u_n\}$ is bounded in $H_\varepsilon $.
\end{proof}
\begin{lemma}\label{2chose1}
Let $\{u_n\}\subset H_\varepsilon$ be a $(PS)_c$ sequence  for $I_\varepsilon$ with
$0<c<\min\left\{ c_\infty,\frac{1}{3}S^{\frac{3}{2}}|K|_{\infty}^{-\frac{1}{2}} \right \} $ and such that $u_n\rightharpoonup  0$, as $n\to\infty$. Then, one of the following conclusions holds.\vspace{0.5ex}\\
\begin{tabular}{rl}
\rm{(\romannumeral1)}& $u_n\to 0$ in $H_\varepsilon$, as $n\to\infty$;\\
\rm{(\romannumeral2)}& there exists a sequence $\{y_n\}\subset \R^3$ and constants $r,\beta>0$ such that
\end{tabular}
$$\liminf\limits_{n\to \infty }\int_{B_r (y_n)} u_n^2 dx\geq\beta>0.$$ 
\end{lemma}
\begin{proof}
Suppose that (ii) does not occur, i.e., for any $r>0$, such that 
$$\limsup\limits_{n\to\infty}\int_{B_r (y)} u_n^2 dx=0.$$
By Lemma \ref{vnsh}, we have $u_n \to 0$ in $L^s(\R^3) $ for $2<s<6$, as $n\to\infty$. Hence, we see that 
\begin{equation}\label{qto0}
\int_{\R^3} Q_i (\varepsilon x)|u_n|^{q_i}dx\to 0,
\end{equation}
as $n\to\infty$.
Moreover, from Lemma \ref{L21} (iv), we have
\begin{equation}
\int_{ \R^3}\phi_{u_n}u_n^2 dx\le S^{-1}|u_n|^4_{12/5}\to0,
\end{equation}
thus, we have $\int_{ \R^3}\phi_{u_n}u_n^2 dx\to0$, as $n\to\infty$. Recalling that $\langle I_{\varepsilon}'(u_n),u_n \rangle=o_n(1)$ , we have 
$$||u_n||^2_\varepsilon=\int_{ \R^3} K(\varepsilon x)|u_n|^{6}dx+o_n(1).$$
It follows from Lemma \ref{psbdd} that $\{u_n\}$ is bounded in $H_\varepsilon$, up to a subsequence, we may assume that there exists $l\geq 0$ such that 
\begin{equation}\label{psto0}
||u_n||^2_\varepsilon\to l,\quad\int_{ \R^3} K(\varepsilon x)|u_n|^{6}dx\to l,
\end{equation}
as $n\to\infty$. If $l=0$, the proof is complete. If $l>0$,  by using \eqref{qto0}--\eqref{psto0} and $I_\varepsilon(u_n)=c+o_n(1)$, we get 
\begin{equation}\label{psc}
\begin{aligned}
c+o_n(1)
=  &\frac{1}{2}||u_n||^2_\varepsilon +\frac{1}{4}\int_{ \R^3}\phi_{u_n}u_n^2 dx\\
& -\sum\limits_{i=1}^{m} \frac{1}{q_i}\int_{ \R^3}Q_i (\varepsilon x)|u_n|^{q_i}dx-\frac{1}{6}\int_{ \R^3} K(\varepsilon x)|u_n|^{6}dx\\
= &\frac{1}{2}||u_n||^2_\varepsilon-\frac{1}{6}\int_{ \R^3} K(\varepsilon x)|u_n|^{6}dx+o_n(1)\\
=  &\frac{1}{3}l+o_n(1).
\end{aligned}
\end{equation} 
By the Sobolev inequality and $(f_4)$, we have that
$$||u_n||^2_\varepsilon\geq\int_{ \R^3}|\nabla u_n|^{2}dx\geq S \left(\int_{\R^3} |u_n|^{6}dx\right)^{\frac{1}{3}}\geq S|K|^{-\frac{1}{3}}_{\infty}\left(\int_{\R^3}K(\varepsilon x) |u_n|^{6}dx\right)^{\frac{1}{3}}.$$
Taking the limit as $n\to\infty$ at the last inequality, we obtain
\begin{equation}\label{lc}
l\geq S^{\frac{3}{2}}|K|_{\infty}^{-\frac{1}{2}}.
\end{equation}
From \eqref{psc} and \eqref{lc}, we get a contradiction to the definition of $c$. Therefore, $l=0$ and the desired conclusion is obtained. 
\end{proof}
\begin{lemma}\label{ps}
Let $\{u_n\}\subset H_\varepsilon$ be a $(PS)_c$ sequence  for $I_\varepsilon$ with $0<c<\min\left\{ c_\infty,\frac{1}{3}S^{\frac{3}{2}}|K|_{\infty}^{-\frac{1}{2}} \right \} $ and $u_n\rightharpoonup 0$ in $ H_\varepsilon$, then  $u_n\to 0$ in $ H_\varepsilon.$
\end{lemma}
\begin{proof}
Assume by contradiction that $u_n\nrightarrow 0$ in $H_\varepsilon$. From Lemma \ref{L23} (i) and (ii), there exists a positive bounded sequence $\{t_n\}$  such that $\{t_nu_n\}\subset\mathcal{N}_\infty$. Then we claim that $\limsup\limits_{n\to \infty }t_n\le1$. Arguing indirectly, there exists $\delta>$ and a subsequence still denoted by $\{t_n\} $, such that $t_n\geq1+\delta$, for all $n\in\mathbb{N}$.
Since $\langle I_{\varepsilon}'(u_n),u_n \rangle=o_n(1)$, we get
\begin{equation}\label{str0}
\begin{aligned}
&||u_n||^2_{\mathcal{D}^{1,2}} +\int_{ \R^3} V(\varepsilon x)|u_n|^{2}dx+\int_{ \R^3}\phi_{u_n}u_n^2 dx \\
= &\sum\limits_{i=1}^{m} \int_{ \R^3}Q_i (\varepsilon x)|u_n|^{q_i}dx + 
\int_{ \R^3} K(\varepsilon x)|u_n|^{6}dx+o_n(1).
\end{aligned}
\end{equation}
From $t_nu_n\in\mathcal{N}_\infty$, we have
\begin{equation}\label{str01}
\begin{aligned}
&t_n^2||u_n||^2_{\mathcal{D}^{1,2}} +t_n^2\int_{ \R^3} V_\infty|u_n|^{2}dx+t_n^4\int_{ \R^3}\phi_{u_n}u_n^2 dx\\
=&\sum\limits_{i=1}^{m}t_n^{q_i} \int_{ \R^3}Q_i^\infty|u_n|^{q_i}dx + t_n^6
\int_{ \R^3} K^\infty|u_n|^{6}dx.
\end{aligned}
\end{equation}
It follows from \eqref{str0} and \eqref{str01} that
\begin{equation}\label{str02}
\begin{aligned}
&o_n(1)+\left(t_n^{2-q_{i_0}}-1\right)||u_n||^2_{\mathcal{D}^{1,2}} 
+\int_{ \R^3}\left( t_n^{2-q_{i_0}}V_\infty-V(\varepsilon x)\right)|u_n|^{2}dx\\
&+\left(t_n^{4-q_{i_0}}-1\right)\int_{ \R^3}\phi_{u_n}u_n^2 dx
-\sum\limits_{i=1}^{i_0-1}\int_{ \R^3}\left( t_n^{q_i-q_{i_0}}Q_i^\infty-Q_i(\varepsilon x)\right)|u_n|^{q_i}dx\\
=&\sum\limits_{i=i_0}^{m}\int_{ \R^3}\left( t_n^{q_i-q_{i_0}}Q_i^\infty-Q_i(\varepsilon x)\right)|u_n|^{q_i}dx
 + \int_{ \R^3}\left(t_n^{6-q_{i_0}} K^\infty-K(\varepsilon x)\right)|u_n|^{6}dx.
\end{aligned}
\end{equation}	
By using the condition $(f_2)$, $4<q_1<q_2<\cdots<q_m<6$, $t_n>1$, and the definition of $ V_\infty$ and $Q_i^\infty$, for any $\sigma>0$, there exists $R=R(\sigma)>0$ such that 
\begin{equation}\label{case21}
V(\varepsilon x)\geq V_\infty-\sigma>t_n^{2-q_{i_0}}V_\infty-\sigma
\end{equation}
and
\begin{equation}
t_n^{q_i-q_{i_0}}Q_i^\infty+\sigma>Q_i^\infty+\sigma\geq Q_i(\varepsilon x),\quad 1\le i\le i_0-1
\end{equation}
and
\begin{equation}
t_n^{q_i-q_{i_0}}Q_i^\infty+\sigma>Q_i^\infty+\sigma\geq Q_i(\varepsilon x),\quad i_0\le i\le n,
\end{equation}
for any $|\varepsilon x|\geq R$. Moreover, it follows from $(f_4)$ that
\begin{equation}\label{case22}
t_n^{6-q_{i_0}}K^\infty\geq K^\infty> K(\varepsilon x),
\end{equation}
for any $x\in\R^3$. Since $u_n\rightharpoonup 0$ in $H_\varepsilon$, we get
\begin{equation}\label{str03}
u_n\to0\hspace{1ex}\text{in}\hspace{1ex}L^q_{loc}(\R^3),\hspace{1ex}q\in[2,6).
\end{equation}
Thus, note that $\{u_n\}$ is bounded in $L^p(\R^3)(2\le p\le 6)$, we deduce from \eqref{str02}--\eqref{str03} that
\begin{equation}\label{327}
\int_{ \R^3}|u_n|^{6}dx<C\sigma.
\end{equation}
On the other hand, Lemma \ref{2chose1} shows that
there exists a sequence $\{y_n\}\subset \R^3$ and constants $r,\beta>0$ such that
\begin{equation}\label{str04}
\liminf\limits_{n\to \infty }\int_{B_r (y_n)} u_n^2 dx\geq\beta>0.
\end{equation}
If we set $v_n(x)=u_n(x+y_n)$, then there exists a non-zero function $v(x)$ such that, up to a subsequence, $v_n\rightharpoonup v$ in $H_\varepsilon$, $v_n\to v$ in $L_{loc}^q(\R^3)$, $q\in[2,6)$, and $v_n\to v$ a.e. in $\R^3$. Moreover, by \eqref{str04}, we have that there exists a subset $\Lambda\subset B_r(0)$ with positive measure such that $v\neq0$ a.e. in $\Lambda$. It follows from Fatou's lemma  and Sobolev inequality that
$$\int_{ \R^3}|u_n|^{6}dx=\int_{ \R^3}|v_n|^{6}dx\geq\int_{ \Lambda}|v|^{6}dx\geq C \int_{ \Lambda}|v_n|^{2}dx>\beta_0>0,$$
for some $\beta_0>0$, which contradicts to \eqref{327} and the claim is true.
\par
We next divide the proof into two separate cases.
\par
Case 1: $\limsup\limits_{n\to \infty }t_n=1.$ 
In this case, there exists a subsequence, still denote by $\{t_n\}$, such that  $t_n\to1$ as $n\to\infty$.
Observe that
\begin{equation}
\begin{aligned}
&I_\infty(t_nu_n)-I_\varepsilon(u_n)\\
=\hspace{1ex}&\frac{t_n^2-1}{2}||u_n||^2_{\mathcal{D}^{1,2}}+\frac{1}{2}\int_{ \R^3}\left( t_n^{2}V_\infty-V(\varepsilon x)\right)u_n^2 dx +\frac{t_n^4-1}{4}\int_{ \R^3}\phi_{u_n}u_n^2 dx\\
 &- \sum\limits_{i=1}^{m} \frac{1}{q_i}\int_{ \R^3}\left( t_n^{q_i}Q_i^\infty-Q_i(\varepsilon x)\right)|u_n|^{q_i}dx-\frac{1}{6}\int_{ \R^3}\left( t_n^{6}K^\infty-K(\varepsilon x)\right)|u_n|^{6}dx.
\end{aligned}
\end{equation}
Arguing as in the proof of the claim above,  fixed $n$ large enough , for any $\sigma>0$, there exists $R=R(\sigma)>0$ such that 
\begin{equation}\label{str1}
V(\varepsilon x)\geq t_n^{2}V_\infty-\sigma,\hspace{1ex}t_n^{q_i}Q_i^\infty+\sigma\geq Q_i(\varepsilon x)(1\le i\le n)
\end{equation}
for any $|\varepsilon x|\geq R$, and
\begin{equation}\label{str2}
t_n^{6}K^\infty\geq  K(\varepsilon x),
\end{equation}
for any $x\in\R^3$. Thus, we deduce from\eqref{str03}, \eqref{str1} and \eqref{str2} that
$$I_\infty(t_nu_n)-I_\varepsilon(u_n)\le o_n(1)+C\sigma.$$
Taking limit in the above inequality, we have $$
\begin{aligned}c+o_n(1)=I_\varepsilon(u_n)\geq& I_\varepsilon(u_n)+c_\infty-I_\infty(t_nu_n)\\
\geq&c_\infty+o_n(1),
\end{aligned}$$
which is a contradiction.
\par
Case 2: $\limsup\limits_{n\to \infty }t_n<1.$ In this case, we suppose that, without loss of generality, $t_n<1$ for all $n\in \mathbb{N}$. From \eqref{case21}--\eqref{str03} and  $\{u_n\}$ is bounded in $L^p(\R^3)(2\le p\le 6)$, for any $\sigma>0$, we have 
\begin{equation}\label{case23}
 \int_{ \R^3}\left( V_\infty-V(\varepsilon x)\right)u_n^2 dx\le C\sigma+o_n(1),
\end{equation}
\begin{equation}
\int_{ \R^3}\left( Q_i^\infty-Q_i(\varepsilon x)\right)|u_n|^{q_i}dx\geq -C\sigma+o_n(1),
\end{equation}
\begin{equation}\label{case24}
\int_{ \R^3}\left( K^\infty-K(\varepsilon x)\right)|u_n|^{6}dx\geq -C\sigma+o_n(1).
\end{equation}
Then, \eqref{case23}--\eqref{case24} imply that 
$$
\begin{aligned}
c_\infty\le & I_\infty(t_nu_n)\\
= &I_\varepsilon(t_nu_n)+\frac{t_n^{2}}{2}\int_{ \R^3}\left( V_\infty-V(\varepsilon x)\right)u_n^2 dx \\
&-\sum\limits_{i=1}^{m} \frac{t_n^{q_i}}{q_i}\int_{ \R^3}\left( Q_i^\infty-Q_i(\varepsilon x)\right)|u_n|^{q_i}dx-\frac{t_n^{6}}{6}\int_{ \R^3}\left( K^\infty-K(\varepsilon x)\right)|u_n|^{6}dx\\
\le& I_\varepsilon(t_nu_n)+C\sigma +o_n(1)\\
\le& I_\varepsilon(u_n)+C\sigma+o_n(1)\\
=&c+C\sigma+o_n(1).
\end{aligned}$$
Let $\sigma\to0$ and $n\to\infty$, we have $c_\infty\le c$, which yields a contradiction ends the proof.
\end{proof}
\begin{lemma}\label{pscc}
	$I_\varepsilon$ satisfies the $(PS)_c$ conditions in $H_\varepsilon$ with  $0<c<\min\left\{c_\infty, \frac{1}{3}S^{\frac{3}{2}}|K|_{\infty}^{-\frac{1}{2}} \right\}$.
\end{lemma}
\begin{proof}
Let $\{u_n\}\subset H_\varepsilon$ be a $(PS)_\varepsilon$ for $I_\varepsilon$ with $0<c<\min\left\{c_\infty, \frac{1}{3}S^{\frac{3}{2}}|K|_{\infty}^{-\frac{1}{2}} \right\}$. By Lemma \ref{psbdd}, we get that $\{u_n\}$ is bounded in $H_\varepsilon$. Then, there exists $u\in H_\varepsilon$ such that $u_n\rightharpoonup u$ in $H_\varepsilon$. By a standard argument, we get $I_\varepsilon'(u_n)\to I_\varepsilon'(u)=0$, i.e., $u$ is a critical point of $I_\varepsilon$. 
Let  $\omega_n=u_n-u$, by Lemma \ref{L22} and Brezis--Lieb Lemma, it is not difficult to see that 
\begin{equation}\label{psc1} I_\varepsilon(\omega_n)=I_\varepsilon(u_n)-I_\varepsilon(u)+o_n(1)=c-I_\varepsilon(u)+o_n(1).
\end{equation}
Under our conditions, we see that 
\begin{equation}\label{psc2}
\begin{aligned}
I_\varepsilon(u)
= & \hspace{0.4em}I_\varepsilon (u)-\frac{1}{q_{i_0}}\langle I_{\varepsilon}'(u),u \rangle\\
=& \left(\frac{1}{2}-\frac{1}{q_{i_0}}\right)||u||^2_\varepsilon +\left(\frac{1}{4}-\frac{1}{q_{i_0}}\right)\int_{ \R^3}\phi_{u}u^2 dx \\
&- \sum\limits_{i=1}^{i_0-1} \left(\frac{1}{q_i}-\frac{1}{q_{i_0}}\right)\int_{ \R^3}Q_i (\varepsilon x)|u|^{q_i}dx+\sum\limits_{i=i_0+1}^{m}\left (\frac{1}{q_{i_0}}-\frac{1}{q_i}\right)\int_{ \R^3}Q_i (\varepsilon x)|u|^{q_i}dx\\
&+\left({\frac{1}{q_{i_0}}-\frac{1}{6}}\right)\int_{ \R^3} K(\varepsilon x)|u|^{6}dx\\
\geq&\left(\frac{1}{2}-\frac{1}{q_{i_0}}\right)||u||^2_\varepsilon\\
\geq&0.
\end{aligned}
\end{equation}
Then, \eqref{psc1} and \eqref{psc2} imply that 
$$I_\varepsilon(\omega_n)\le c.$$
It follows from Lemma \ref{ps} that $\omega_n\to 0$ in $H_\varepsilon$, then $u_n\to u$ in $H_\varepsilon$ and the proof is complete. 
\end{proof}
\begin{lemma}\label{L39}
Problem \eqref{22} has at least a positive ground state solution in $H_\varepsilon$.
\end{lemma}
\begin{proof}
 From Lemma \ref{L24}, we have that $I_{\varepsilon}$ satisfies the mountain pass geometry structure. Moreover, using a version of the mountain pass theorem without $(PS)$ condition, there exists a $(PS)_{c_\varepsilon}$ sequence  $\{u_n\}\subset H_\varepsilon$ for $I_\varepsilon$. Lemma \ref{cric1} and Lemma \ref{cric2} imply that  $0<c_\varepsilon<\min\left\{c_\infty, \frac{1}{3}S^{\frac{3}{2}}|K|_{\infty}^{-\frac{1}{2}} \right\}$ for $\varepsilon>0$ small enough. Therefore, with the aim of Lemma \ref{pscc}, we conclude that there exists $u_\varepsilon\in H_\varepsilon$ such that
 $$I_\varepsilon(u_\varepsilon)=c_\varepsilon,\hspace{1ex} \text{and} \hspace{1ex}I_\varepsilon'(u_\varepsilon)=0.$$
It follows from \eqref{3eq} that $u_\varepsilon$ is a ground state solution of \eqref{22}. If we denote $u_\varepsilon^{\pm}=\max\{\pm u_\varepsilon, 0\}$, and replace $I_\varepsilon$ by the following functional
$$ I^{+}_\varepsilon(u_\varepsilon)=\frac{1}{2}||u_\varepsilon||^2_\varepsilon +\frac{1}{4}\int_{ \R^3}\phi_{u_\varepsilon}u_\varepsilon^2 dx - \sum\limits_{i=1}^{m} \frac{1}{q_i}\int_{ \R^3}Q_i (\varepsilon x)|u_\varepsilon^+|^{q_i}dx-\frac{1}{6}\int_{ \R^3} K(\varepsilon x)|u_\varepsilon^+|^{6}dx.$$
Repeating the above proof and calculations, we have
$$0=\langle  (I^{+})'_\varepsilon(u_\varepsilon), u_\varepsilon^-\rangle=||u_\varepsilon^-||^2_\varepsilon +\int_{ \R^3}\phi_{u_\varepsilon^-}(u_\varepsilon^-)^2 dx\geq ||u_\varepsilon^-||^2_\varepsilon,$$
which implies $u_\varepsilon\geq 0$ in $\R^3$. The strong maximum principle implies that $u_\varepsilon(x)>0$ for all $x\in\R^3$. The proof is complete.
\end{proof}
 \vskip2mm
 {\section{Concentration of positive ground state solutions}\label{sec4}}
 \setcounter{equation}{0}
 \vskip2mm
In this section, in order to discuss the concentration behavior of ground state solutions $v_\varepsilon$ of \eqref{1} with $\varepsilon\to0$, we will consider the family $u_\varepsilon(x)=v_\varepsilon(\varepsilon x)$, which is a family of positive ground state  solutions of \eqref{22}.
\begin{lemma}\label{L41}
	There exist  $\varepsilon_*>0$, $\{y_\varepsilon\}\subset\R^3$ and $r$, $\beta>0$, such that
	$$\int_{B_r (y_\varepsilon)} |u_\varepsilon|^2 dx\geq\beta,\hspace{1ex}\text{for all}\hspace{1ex}\varepsilon\in(0,\varepsilon_*).$$ 
\end{lemma}
\begin{proof}
Arguing by contradiction, suppose that the lemma does not hold. Then, there exists a sequence $\varepsilon_n\to0$, such that for all $r>0$,
$$\lim\limits_{n\to\infty}\sup\limits_{y\in \R^3}\int_{B_r (y)}|u_{\varepsilon_n}|^2dx=0.$$
Then, repeating the arguments employed in the proof of Lemma \ref{2chose1}, we can obtain that
$$c_{\varepsilon_n}\geq \frac{1}{3}S^{\frac{3}{2}}|K|_{\infty}^{-\frac{1}{2}},$$
which is a contradiction with $c_\varepsilon<\frac{1}{3}S^{\frac{3}{2}}|K|_{\infty}^{-\frac{1}{2}}.$ The proof is complete.
\end{proof}
\begin{lemma}\label{L42}
The family $\{\varepsilon y_\varepsilon\}$ is bounded as $\varepsilon\to0$. Moreover, assume that $\varepsilon_n y_{\varepsilon_n}\to x_0$ under a choice of subsequence, then we have $x_0\in \mathcal{G}$, i.e., $$G(x_0)=\inf\limits_{s\in \R^3}G(s).$$ 
\end{lemma}
\begin{proof}
	For the sake of simplicity, we denote $y_n=y_{\varepsilon_n}$,  $u_n(x)=u_{\varepsilon_n }(x)$.
	Arguing by contradiction, suppose that there is a sequence $\varepsilon_n\to 0$ such that $ |\varepsilon_n y_n|\to\infty$, as $n\to\infty$.
	Set $\widetilde{u}_n(x)=u_n(x+y_n)$, where $\widetilde{u}_n(x)=\widetilde{u}_{\varepsilon_n}(x)$. It follows from Lemma \ref{L41} that 
	\begin{equation}\label{ineq41}
\int_{B_r (0)} |\widetilde{u}_n|^2 dx\geq\beta,\hspace{1ex}\text{for all}\hspace{1ex}n\in \mathbb{N}.
	\end{equation}
Then, $\widetilde{u}_n$ is a ground state solution of
\begin{equation*}
-\Delta\widetilde{u}_n+\phi_{\widetilde{u}_n} \widetilde{u}_n+V(\varepsilon_n x+\varepsilon_ny_n)\widetilde{u}_n=\sum\limits_{i=1}^{m} Q_i(\varepsilon_n x+\varepsilon_ny_n)|\widetilde{u}_n|^{q_i-2}\widetilde{u}_n+K(\varepsilon_n x+\varepsilon_ny_n)|\widetilde{u}_n|^{4}\widetilde{u}_n,
\end{equation*}
and $||\widetilde{u}_n||_{\varepsilon_n}=||u_n||_{\varepsilon_n}$ is bounded in $\R$. Moreover, up to a subsequence,  we may assume that $\widetilde{u}_n\rightharpoonup\widetilde{u}$ in $H_\varepsilon$ with $\widetilde{u}\neq0$ and $\widetilde{u}\geq0$ . It follows that
\begin{equation}\label{41}
\begin{aligned}
&\int_{ \R^3} \nabla \widetilde{u}_n \nabla \varphi dx +\int_{ \R^3} V(\varepsilon_n x+\varepsilon_ny_n)\widetilde{u}_n\varphi dx+\int_{ \R^3}\phi_{\widetilde{u}_n}\widetilde{u}_n\varphi dx \\
= &\sum\limits_{i=1}^{m} \int_{ \R^3}Q_i(\varepsilon_n x+\varepsilon_ny_n)|\widetilde{u}_n|^{q_i-1}\varphi dx + 
\int_{ \R^3}K(\varepsilon_n x+\varepsilon_ny_n)|\widetilde{u}_n|^{5}\varphi dx,
\end{aligned}
\end{equation}
for any $ \varphi\in H^1(\R^3)$. Without loss of generality, we may asuume that $V(\varepsilon_ny_n)\to V_\infty$ and $Q_i(\varepsilon_ny_n)\to Q_i^\infty$ $(1\le i\le m)$ and $K(\varepsilon_ny_n)\to K^\infty$, as $n\to\infty$. Under our assumptions on $V$, $K$, and $Q_i(1\le i\le m)$, we get that  $V$, $K$, and $Q_i(1\le i\le m)$ are uniformly continuous, which implies that 
\begin{equation}
V(\varepsilon_n x+\varepsilon_ny_n)\to V_\infty\quad\text{and}\quad K(\varepsilon_n x+\varepsilon_ny_n)\to K^\infty,
\end{equation}
and
\begin{equation}\label{43}
Q_i(\varepsilon_n x+\varepsilon_ny_n)\to Q_i^\infty(1\le i\le m),
\end{equation}
as $n\to\infty$ uniformly on bounded sets of $\R^3$.
Using the weak limit of  $\widetilde{u}_n$ and \eqref{41}--\eqref{43}, we get
\begin{equation*}
\begin{aligned}
&\int_{ \R^3} \nabla \widetilde{u} \nabla \varphi dx +\int_{ \R^3} V_\infty \widetilde{u}\varphi dx+\int_{ \R^3}\phi_{\widetilde{u}}\widetilde{u}\varphi dx \\
= &\sum\limits_{i=1}^{m} \int_{ \R^3}Q_i^\infty|\widetilde{u}|^{q_i-1}\varphi dx + 
\int_{ \R^3}K^\infty|\widetilde{u}|^{5}\varphi dx,
\end{aligned}
\end{equation*}
for any $ \varphi\in H^1(\R^3)$.
Which implies that $\widetilde{u}\in \mathcal{N}_\infty$, i.e., $\langle I_{\infty}'(\widetilde{u}),\widetilde{u} \rangle=0.$ We deduce  from Lemma \ref{cric1} and Fatou's Lemma that
\begin{equation}
\begin{aligned}
c_\infty\le& I_{\infty}(\widetilde{u})-\frac{1}{q_{i_0}}\langle I_{\infty}'(\widetilde{u}),\widetilde{u} \rangle\\
=& \left(\frac{1}{2}-\frac{1}{q_{i_0}}\right)\int_{ \R^3}\left(|\nabla \widetilde{u}|^2+V_\infty|\widetilde{u}|^{2}\right) dx +\left(\frac{1}{4}-\frac{1}{q_{i_0}}\right)\int_{ \R^3}\phi_{\widetilde{u}}\widetilde{u}^2 dx \\
&- \sum\limits_{i=1}^{i_0-1} \left(\frac{1}{q_i}-\frac{1}{q_{i_0}}\right)\int_{ \R^3}Q_i^\infty|\widetilde{u}|^{q_i}dx+\sum\limits_{i=i_0+1}^{m}\left (\frac{1}{q_{i_0}}-\frac{1}{q_i}\right)\int_{ \R^3}Q_i^\infty|\widetilde{u}|^{q_i}dx\\
&+\left({\frac{1}{q_{i_0}}-\frac{1}{6}}\right)\int_{ \R^3} K^\infty|\widetilde{u}|^{6}dx\\
\le&\liminf\limits_{n\to \infty }  \left(\frac{1}{2}-\frac{1}{q_{i_0}}\right)\int_{ \R^3}\left(|\nabla \widetilde{u}_n|^2+V(\varepsilon_n x+\varepsilon_ny_n)|\widetilde{u}_n|^{2}\right) dx +\left(\frac{1}{4}-\frac{1}{q_{i_0}}\right)\int_{ \R^3}\phi_{\widetilde{u}_n}\widetilde{u}_n^2 dx \\
&- \sum\limits_{i=1}^{i_0-1} \left(\frac{1}{q_i}-\frac{1}{q_{i_0}}\right)\int_{ \R^3}Q_i(\varepsilon_n x+\varepsilon_ny_n)|\widetilde{u}_n|^{q_i}dx\\
&+\sum\limits_{i=i_0+1}^{m}\left (\frac{1}{q_{i_0}}-\frac{1}{q_i}\right)\int_{ \R^3}Q_i(\varepsilon_n x+\varepsilon_ny_n)|\widetilde{u}_n|^{q_i}dx\\
&+\left({\frac{1}{q_{i_0}}-\frac{1}{6}}\right)\int_{ \R^3} K(\varepsilon_n x+\varepsilon_ny_n)|\widetilde{u}_n|^{6}dx\\
\le&\limsup\limits_{n\to \infty } \left(I_{\varepsilon_n}(\widetilde{u}_n)-\frac{1}{q_{i_0}}\langle I_{\varepsilon_n}'(\widetilde{u}_n),\widetilde{u}_n \rangle\right)\\
=&\limsup\limits_{n\to \infty }c_{\varepsilon_n}
\le c_0,
\end{aligned}
\end{equation}
which contradicts the fact that $c_0>c_\infty$. Hence, $\{\varepsilon y_\varepsilon\}$ is bounded as $\varepsilon\to0$. Note that if $\varepsilon_n y_n\to x_0$ under a choice of subsequence, and $V$, $Q_i(1\le i\le m ) $, $K$ are uniformly continuous, we have 
\begin{equation}\label{46}
V(\varepsilon_n x+\varepsilon_ny_n)\to V(x_0)\quad\text{and}\quad K(\varepsilon_n x+\varepsilon_ny_n)\to K(x_0),
\end{equation}
and
\begin{equation}\label{477}
 Q_i(\varepsilon_n x+\varepsilon_ny_n)\to Q_i(x_0)(1\le i\le m),
\end{equation}
as $n\to\infty$ uniformly on bounded sets of $\R^3$. Taking $\varphi=\widetilde{u}$ and limit as $n\to\infty$ in the above Eq. \eqref{41}, we get 
\begin{equation*}
\begin{aligned}
||\widetilde{u}||^2_{\mathcal{D}^{1,2}} +\int_{ \R^3} V(x_0)|\widetilde{u}^{2}dx+\int_{ \R^3}\phi_{\widetilde{u}}\widetilde{u}^2 dx 
= \sum\limits_{i=1}^{m} \int_{ \R^3}Q_i(x_0)|\widetilde{u}|^{q_i}dx + 
\int_{ \R^3} K(x_0)|\widetilde{u}|^{6}dx,
\end{aligned}
\end{equation*}
which implies that $\widetilde{u}\in \mathcal{N}^{x_0}$, i.e., $\langle (I^{x_0})'(\widetilde{u}),\widetilde{u} \rangle=0.$
We deduce  from Lemma \ref{cric1} and Fatou's Lemma that
\begin{equation}\label{47}
\begin{aligned}
c_0=& \inf\limits_{s\in\R^3}G(s)\le G(x_0)\\
 \le& I^{x_0}(\widetilde{u} )-\frac{1}{q_{i_0}}\langle (I^{x_0})'(\widetilde{u}),\widetilde{u} \rangle\\
=& \left(\frac{1}{2}-\frac{1}{q_{i_0}}\right)\int_{ \R^3}\left(|\nabla \widetilde{u}|^2+V(x_0)|\widetilde{u}|^{2}\right) dx +\left(\frac{1}{4}-\frac{1}{q_{i_0}}\right)\int_{ \R^3}\phi_{\widetilde{u}}\widetilde{u}^2 dx \\
&- \sum\limits_{i=1}^{i_0-1} \left(\frac{1}{q_i}-\frac{1}{q_{i_0}}\right)\int_{ \R^3}Q_i(x_0)|\widetilde{u}|^{q_i}dx+\sum\limits_{i=i_0+1}^{m}\left (\frac{1}{q_{i_0}}-\frac{1}{q_i}\right)\int_{ \R^3}Q_i(x_0)|\widetilde{u}|^{q_i}dx\\
&+\left({\frac{1}{q_{i_0}}-\frac{1}{6}}\right)\int_{ \R^3} K(x_0)|\widetilde{u}|^{6}dx\\
\le&\liminf\limits_{n\to \infty }  \left(\frac{1}{2}-\frac{1}{q_{i_0}}\right)\int_{ \R^3}\left(|\nabla \widetilde{u}_n|^2+V(\varepsilon_n x+\varepsilon_ny_n)|\widetilde{u}_n|^{2}\right) dx +\left(\frac{1}{4}-\frac{1}{q_{i_0}}\right)\int_{ \R^3}\phi_{\widetilde{u}_n}\widetilde{u}_n^2 dx \\
&- \sum\limits_{i=1}^{i_0-1} \left(\frac{1}{q_i}-\frac{1}{q_{i_0}}\right)\int_{ \R^3}Q_i(\varepsilon_n x+\varepsilon_ny_n)|\widetilde{u}_n|^{q_i}dx\\
&+\sum\limits_{i=i_0+1}^{m}\left (\frac{1}{q_{i_0}}-\frac{1}{q_i}\right)\int_{ \R^3}Q_i(\varepsilon_n x+\varepsilon_ny_n)|\widetilde{u}_n|^{q_i}dx\\
&+\left({\frac{1}{q_{i_0}}-\frac{1}{6}}\right)\int_{ \R^3} K(\varepsilon_n x+\varepsilon_ny_n)|\widetilde{u}_n|^{6}dx\\
\le&\limsup\limits_{n\to \infty } \left(I_{\varepsilon_n}(\widetilde{u}_n)-\frac{1}{q_{i_0}}\langle I_{\varepsilon_n}'(\widetilde{u}_n),\widetilde{u}_n \rangle\right)\\
=&\limsup\limits_{n\to \infty }c_{\varepsilon_n}\\
\le&c_0,
\end{aligned}
\end{equation}
it follows that $G(x_0)=\inf\limits_{s\in \R^3}G(s)$. The proof is complete.
\end{proof}
\begin{lemma}\label{L43}
$\widetilde{u}_n \to \widetilde{u}$ in $H_\varepsilon$ as $n\to\infty$. Furthermore, there exist $C>0$ and $\varepsilon_*>0$ such that $|\widetilde{u}_n|_\infty\le C $ and 
$$ \lim\limits_{|x|\to\infty}\widetilde{u}_{\varepsilon}(x)=0 \hspace{1ex}\text{uniformly on} \hspace{1ex}\varepsilon\in (0,\varepsilon_*).$$
\end{lemma}
\begin{proof}
Due to the Brezis--Lieb lemma and \eqref{46}--\eqref{477}, we get
$$I_{\varepsilon_n}(\widetilde{u}_n - \widetilde{u}_)=I_{\varepsilon_n}(\widetilde{u}_n)-I^{x_0}(\widetilde{u})+o_n(1).$$
Since $\varepsilon_n y_{\varepsilon_n}\to x_0$ and \eqref{47}, we have 
$$\lim\limits_{n\to\infty}I_{\varepsilon_n}(\widetilde{u}_n - \widetilde{u})=0. $$ 
Similarly, we also get 
$$\lim\limits_{n\to\infty}\langle I_{\varepsilon_n}'(\widetilde{u}_n - \widetilde{u}),\widetilde{u}_n - \widetilde{u} \rangle=0.$$
Consequently,
$$\left(\frac{1}{2}-\frac{1}{q_{i_0}}\right)||\widetilde{u}_n - \widetilde{u}||^2 \le\lim\limits_{n\to \infty } \left(I_{\varepsilon_n}(\widetilde{u}_n- \widetilde{u})-\frac{1}{q_{i_0}}\langle I_{\varepsilon_n}'(\widetilde{u}_n- \widetilde{u}),\widetilde{u}_n- \widetilde{u} \rangle\right)=0,$$
which implies that $\widetilde{u}_n \to \widetilde{u}$ in $H_\varepsilon$ as $n\to\infty$.
 From Lemma \ref{L42}, we know that the sequence $\widetilde{u}_n $ satisfies
\begin{equation}\label{49}
-\Delta\widetilde{u}_n+\left(\phi_{\widetilde{u}_n}+V(\varepsilon_n x+\varepsilon_ny_n)-K(\varepsilon_n x+\varepsilon_ny_n)|\widetilde{u}_n|^{4}\right) \widetilde{u}_n=\sum\limits_{i=1}^{m} Q_i(\varepsilon_n x+\varepsilon_ny_n)|\widetilde{u}_n|^{q_i-2}\widetilde{u}_n.
\end{equation}
By Lemma \ref{L21} (iv) and \eqref{ineq41}, we obtain that $0<\phi_{\widetilde{u}_n} <C$, then $\phi_{\widetilde{u}_n}+V(\varepsilon_n x+\varepsilon_ny_n)\in L^{\infty}_{loc}(\R^3)$. Moreover, since $K(\varepsilon_n x+\varepsilon_ny_n)|\widetilde{u}_n|^{4}\in L^\frac{3}{2}(\R^3)$ and $4<q_1<q_2<\cdots<q_m<6$, using a result in \cite[Proposition 3.3]{hezou} or \cite{BK}, we have  $\widetilde{u}_n\in L^{t}(\R^3)$ for all $t\geq 2$. Furthermore, $\widetilde{u}_n $ satisfies
\begin{equation}
\begin{aligned}
-\Delta\widetilde{u}_n&\le
-\Delta\widetilde{u}_n+\left(\phi_{\widetilde{u}_n}+V(\varepsilon_n x+\varepsilon_ny_n)\right) \widetilde{u}_n\\
&=g_n(x):=\sum\limits_{i=1}^{m} Q_i(\varepsilon_n x+\varepsilon_ny_n)|\widetilde{u}_n|^{q_i-2}\widetilde{u}_n+K(\varepsilon_n x+\varepsilon_ny_n)|\widetilde{u}_n|^{4} \widetilde{u}_n,
\end{aligned}
\end{equation}
where $g_n(x)\in L^{\frac{s}{2}}(\R^3)$, for some $s>3$. Applying a result of \cite{NS} or \cite[Proposition 3.4]{hezou}, we can obtain that
$$\sup\limits_{x\in B_r(y)}\widetilde{u}_n(x)\le C\left(|\widetilde{u}_n|_{L^2(B_{2r}(y))}+|\widetilde{u}_n|_{L^{\frac{s}{2}}(B_{2r}(y))}\right),\hspace{1ex}\text{for any}\hspace{1ex} y\in \R^3,$$
which implies that $|\widetilde{u}_n|_{\infty}\le C$ and 
$$ \lim\limits_{|x|\to\infty}\widetilde{u}_n(x)=0 \hspace{1ex}\text{uniformly on} \hspace{1ex}n\in \mathbb{N}.$$
Consequently, there exists $\varepsilon_*>0$ such that
\begin{equation}\label{411}
\lim\limits_{|x|\to\infty}\widetilde{u}_{\varepsilon}(x)=0 \hspace{1ex}\text{uniformly on} \hspace{1ex}\varepsilon\in (0,\varepsilon_*).
\end{equation}
The proof is complete.
\end{proof}
In order to see the exponential decay of solutions $u_\varepsilon$, it is enough to show the following result about  $\widetilde{u}_\varepsilon$.
\begin{lemma}\label{L44}
There exist constants $C>0$ and $\mu >0$ such that
$$\widetilde{u}_\varepsilon\le C e^{-\mu|x|} \hspace{1ex}\text{for all}\hspace{1ex}x\in\R^3.$$
\end{lemma}
\begin{proof}
We borrow an idea from \cite[Lemma 3.11]{hezou}. It follows from \eqref{411} that there exists $R>0$ such that  
\begin{equation}
\sum\limits_{i=1}^{m} Q_i(x_0)|\widetilde{u}_\varepsilon|^{q_i-2}+K(x_0)|\widetilde{u}_\varepsilon|^{4}\le \frac{V(x_0)}{2}\hspace{1ex}\text{for all} \hspace{1ex} |x|> R\hspace{1ex}\text{and}\hspace{1ex}\forall\varepsilon\in(0,\varepsilon_*).
\end{equation}
Set $\tau (x) =C e^{-\mu|x|}$ such that $\mu^2<\frac{V(x_0)}{2}$ and $\tau(R)\geq \widetilde{u}_\varepsilon$, we can obtain that
\begin{equation}\label{tau}
\nabla\tau(x)=\mu^2\tau(x).
\end{equation}
By Lemma \ref{L21} (iv) and \eqref{ineq41}, we can get 
\begin{equation}\label{tau1}
\begin{aligned}
-\Delta\widetilde{u}_\varepsilon+V(x_0)\widetilde{u}_\varepsilon
&<-\Delta\widetilde{u}_\varepsilon+V(x_0)\widetilde{u}_\varepsilon+\phi_{\widetilde{u}_\varepsilon}\widetilde{u}_\varepsilon\\
&=\sum\limits_{i=1}^{m} Q_i(x_0)|\widetilde{u}_\varepsilon|^{q_i-2}\widetilde{u}_\varepsilon+K(x_0)|\widetilde{u}_\varepsilon|^{4}\widetilde{u}_\varepsilon\\
&\le \frac{V(x_0)}{2}\widetilde{u}_\varepsilon \hspace{1ex}\text{for} \hspace{1ex} |x|> R.
\end{aligned}
\end{equation}
Let $\tau_\varepsilon=\tau-\widetilde{u}_\varepsilon$, combining \eqref{tau} and \eqref{tau1}, we have
\begin{equation}
\begin{cases}
-\Delta\tau_\varepsilon +\frac{V(x_0)}{2}\tau_\varepsilon>0,& |x|>R,\\
\tau_\varepsilon\geq0,&|x|=R,\\
\lim\limits_{|x|\to\infty}\tau_\varepsilon(x)=0.
\end{cases}
\end{equation}
The strong maximum principle implies that $\tau_\varepsilon\geq0$ in $|x|\geq R$. It follows that $$\widetilde{u}_\varepsilon\le C e^{-\mu|x|}\hspace{1ex}\text{for all} \hspace{1ex} |x|\geq R\hspace{1ex}\text{and}\hspace{1ex}\forall\varepsilon\in(0,\varepsilon_*).$$ The proof is complete.
\end{proof}

 {\section{Proof of Theorem \ref{the1}--\ref{the2}}\label{sec5}}
 \setcounter{equation}{0}
 \vskip2mm
 In this section, we will prove the existence, concentration and exponential decay of $v_\varepsilon$ in Theorem \ref{the1}--\ref{the2}.
  \vskip4mm
 {\subsection{Proof of Theorem \ref{the1}}\label{subsec51}}
   \vskip2mm
By Lemma \ref{L39}, $u_\varepsilon$ is a positive ground state solution of \eqref{22}. Then, $v_\varepsilon(x)=u_\varepsilon(\frac{x}{\varepsilon})$ is a positive ground state solution of \eqref{1}. Suppose $z_n$ denotes a maximum point of $\widetilde{u}_n$, it follows from Lemma \ref{L43} that $z_n$ is a bounded sequence in $\R^3$. Then, there exists $R>0$, such that $z_n\in B_R(0)$. Thus, the global maximum point of $u_n$ is $z_n+y_n$. 
Using the boundedness of $\{z_n\}$ and Lemma \ref{L42}, we have
$$\lim\limits_{n\to \infty }\varepsilon_n(z_n+y_n)=x_0.$$ 
Then, note the relations of  $\widetilde{u}_\varepsilon$, $u_\varepsilon$ and $v_\varepsilon$, we have
\begin{equation}\label{51}
v_\varepsilon(\varepsilon x+\varepsilon z_\varepsilon +\varepsilon y_\varepsilon)=u_\varepsilon( x+ z_\varepsilon + y_\varepsilon)=\widetilde{u}_\varepsilon( x+ z_\varepsilon).
\end{equation}
Set $\eta_\varepsilon(x)=v_\varepsilon(\varepsilon x+x_\varepsilon)$, where $x_\varepsilon\to x_0,$ as $\varepsilon\to0$.  It follows from Lemma \ref{L41} and \eqref{51} that $\eta\neq0$.  By the similar arguments to Lemma \ref{L42} and Lemma \ref{L43}, we can obtain that $\eta_\varepsilon\to \eta $ in $H_\varepsilon $ and $\eta$ is a ground state solution of
\begin{equation*}
\begin{aligned}
-\Delta u+\phi_u u+V(x_0)u=\sum\limits_{i=1}^{m} Q_i(x_0)|u|^{q_i-2}u+K(x_0)|u|^{4}u,\hspace{1ex}x\in\R^3.
\end{aligned}
\end{equation*} 
From Lemma \ref{L44},  we have
$$v_\varepsilon(x)=u_\varepsilon(\frac{x}{\varepsilon})=\widetilde{u}_\varepsilon(\frac{x-\varepsilon y_\varepsilon}{\varepsilon})\le C exp\left({-\frac{\mu}{\varepsilon}|x-x_\varepsilon|}\right), $$
for all $x\in\R^3.$ The proof is complete. \hspace{9.15cm}$\Box$
  \vskip4mm
 {\subsection{Proof of Theorem \ref{the2}}}
 \vskip2mm
 In this subsection, we will consider that the electronic potential $h(x)$ satisfies $(H_1)$.
 By the similar strategy to Lemma \ref{L39}, we can obtain that problem \eqref{1} has a positive ground state solution, still denote by $v_\varepsilon$. Then, the following energy functional 
 \begin{equation}
 \begin{aligned}
  \mathcal{I}_\varepsilon(u)=&\frac{1}{2}||u||^2_\varepsilon +\frac{1}{4}\int_{ \R^3}\left[\frac{1}{|x|}*h(\varepsilon x)u^2\right]h(\varepsilon x)u^2\rm{d} x\\
   &- \sum\limits_{i=1}^{m} \frac{1}{q_i}\int_{ \R^3}Q_i (\varepsilon x)|u|^{q_i}dx-\frac{1}{6}\int_{ \R^3} K(\varepsilon x)|u|^{6}dx,
 \end{aligned}
 \end{equation} 
 has a critical point $u_\varepsilon$. Next, we want to show the following  claims.
 \vskip2mm
 \par\noindent
 \begin{claim}\
 Suppose that $\varepsilon_n y_n\to \infty$, as $n\to\infty$, then 
 $$\int_{R^3}\left[\frac{1}{|x|}*h(\varepsilon_n x+\varepsilon_ny_n)\widetilde{u}_n^2\right]h(\varepsilon_n x+\varepsilon_ny_n)\widetilde{u}_n\varphi dx\to0,$$
 for any $ \varphi\in H^1(\R^3)$, as $n\to\infty.$
 \end{claim}
  \begin{proof}
  	From assumption $(H_1),$ we know that $h(x)$ is bounded in $\R^3$.
  	Since $\{\widetilde{u}_n\}$ is bounded in $H_\varepsilon$, by H\"{o}lder inequality, we have
  	\begin{equation}\label{53}
  	\begin{aligned}
\left|\frac{1}{|x|}*h(\varepsilon_n x+\varepsilon_ny_n)\widetilde{u}_n^2\right|  
 & =\left|\int_{\R^3}h(\varepsilon_n y+\varepsilon_ny_n)\frac{\widetilde{u}_n^{2}(y)}{|x-y|}dy\right|\\
  &\le C_1\int_{\R^3}\frac{\widetilde{u}_n^{2}(y)}{|x-y|}dy\\
  &=C_1\int_{|x-y|\le1}\frac{\widetilde{u}_n^{2}(y)}{|x-y|}dy+C_1\int_{|x-y|\geq1}\frac{\widetilde{u}_n^{2}(y)}{|x-y|}dy.\\
  &\le C_1\int_{|x-y|\le1}\frac{\widetilde{u}_n^{2}(y)}{|x-y|}dy+C_2\\
  &\le C_1\left(\int_{|x-y|\le1}\widetilde{u}_n^{4}(y)dy\right)^{\frac{1}{2}}
 \left(\int_{|x-y|\le1}\frac{1}{|x-y|^2}dy\right)^{\frac{1}{2}}+C_2\\
 &\le C_3 \left(\int_{|r|\le1}dr\right)^{\frac{1}{2}}+C_2\\
 &\le C.
  	\end{aligned}
  	\end{equation}
 Since  $\varepsilon_n y_n\to \infty$, as $n\to\infty$, it follows from $(H_1)$ that $h(\varepsilon_n x+\varepsilon_ny_n)\to0$, as $n\to\infty$. Note that
  	\begin{equation}\label{54}
  	\begin{aligned}
  	\widetilde{u}_n\rightharpoonup\widetilde{u} &\hspace{2ex}\text{in} \hspace{1ex}H_\varepsilon;\\
  \widetilde{u}_n\rightharpoonup\widetilde{u}	&\hspace{2ex}\text{in}\hspace{1ex} L^p(\R^3),\hspace{1ex}2\le p\le 6;\\
  	\widetilde{u}_n\to\widetilde{u}&\hspace{2ex}\text{a.e.}\hspace{1ex}\text{on}\hspace{1ex}\R^3.
  	\end{aligned}
  	\end{equation}
  	We can obtain that $h(\varepsilon_n x+\varepsilon_ny_n)\widetilde{u}_n\to0$ a.e. on $\R^3$ and $h(\varepsilon_n x+\varepsilon_ny_n)\widetilde{u}_n$ is bounded in $L^2(\R^3)$. It follows that $h(\varepsilon_n x+\varepsilon_ny_n)\widetilde{u}_n\rightharpoonup0$ in $L^2(\R^3)$. Then, we have 
  	\begin{equation}\label{55}
  	\begin{aligned}
  	\int_{\R^3}h(\varepsilon_n x+\varepsilon_ny_n)\widetilde{u}_n\varphi dx\to0,
  	\end{aligned}
  	\end{equation}
as $n\to\infty$. Combining \eqref{53} and \eqref{55}, the desired conclusion  is obtained.
  \end{proof}
\begin{claim}
Suppose that $\varepsilon_n y_n\to x_0$, as $n\to\infty$, then
$$\begin{aligned}
&\int_{\R^3}\left[\frac{1}{|x|}*h(\varepsilon_n x+\varepsilon_ny_n)\widetilde{u}_n^2\right]h(\varepsilon_n x+\varepsilon_ny_n)\widetilde{u}_n\varphi dx\\
=&h(x_0)^2\int_{ \R^3}\left[\frac{1}{|x|}*\widetilde{u}^2\right]\widetilde{u}\varphi dx+o_n(1)\\
=&o_n(1),
\end{aligned}$$
 for any $ \varphi\in H^1(\R^3)$.
\end{claim}
\begin{proof}
For simplicity, we denote $h(\varepsilon_n x+\varepsilon_ny_n)$ by $\widetilde{h}_n(x)$. Then, from assumption $(H_1)$, we have
\begin{equation}\label{56}
\widetilde{h}_n(x)\to h(x_0),
\end{equation}
as $n\to\infty$ uniformly on bounded sets of $\R^3$. Combining \eqref{54} and \eqref{56}, we have
$\widetilde{h}_n(x)\widetilde{u}_n\to h(x_0)\widetilde{u}$ a.e. on $\R^3$ and $\widetilde{h}_n(x)\widetilde{u}_n$ is bounded in $L^2(\R^3)$. It follows that $\widetilde{h}_n(x)\widetilde{u}_n\rightharpoonup h(x_0)\widetilde{u}$ in $L^2(\R^3)$. Then, we have 
\begin{equation}\label{57}
\left|\int_{\R^3}\left(\widetilde{h}_n(x)\widetilde{u}_n-h(x_0) \widetilde{u}\right)\varphi dx\right|\to0,
\end{equation}
 for any $ \varphi\in H^1(\R^3)$, as $n\to\infty.$ By \eqref{53} and \eqref{57}, we have 
\begin{equation}\label{58}
\left|\int_{ \R^3}\left[\frac{1}{|x|}*\widetilde{h}_n(x)\widetilde{u}_n^2\right]\left(\widetilde{h}_n(x)\widetilde{u}_n-h(x_0) \widetilde{u} \right)\varphi dx\right|\to0,
\end{equation}
 for any $ \varphi\in H^1(\R^3)$, as $n\to\infty$. On the other hand, using \eqref{54} and \eqref{56}, we can deduce that
 $\widetilde{h}_n(x)\widetilde{u}_n^2\rightharpoonup h(x_0)\widetilde{u}^2$ in $L^{\frac{12}{5}}(\R^3)$. It follows from Lemma \ref{L21} (iv) that 
 \begin{equation*}
 \frac{1}{|x|}*\widetilde{h}_n(x)\widetilde{u}_n^2\rightharpoonup\frac{1}{|x|}*h(x_0)\widetilde{u}^2 \hspace{2ex}\text{in} \hspace{1ex}\mathcal{D}^{1,2}(\R^3)
 \end{equation*} 
and, then, in $L^6(\R^3)$. Then, we have 
\begin{equation}
\left|\int_{ \R^3}\left[\frac{1}{|x|}*\left(\widetilde{h}_n(x)\widetilde{u}_n^2-h(x_0)\widetilde{u}^2\right) \right]\widetilde{u}\varphi dx\right|\to0,
\end{equation}
 for any $ \varphi\in H^1(\R^3)$, as $n\to\infty.$
Note that 	
\begin{equation}\label{510}
\begin{aligned}
&\left|\int_{\R^3}\left[\frac{1}{|x|}*\widetilde{h}_n(x)\widetilde{u}_n^2\right]\widetilde{h}_n(x)\widetilde{u}_n\varphi dx-h(x_0)^2\int_{ \R^3}\left[\frac{1}{|x|}*\widetilde{u}^2\right]\widetilde{u}\varphi dx\right|\\
\le&\left|\int_{ \R^3}\left[\frac{1}{|x|}*\widetilde{h}_n(x)\widetilde{u}_n^2\right]\left(\widetilde{h}_n(x)\widetilde{u}_n-h(x_0) \widetilde{u} \right)\varphi dx\right|\\
&+\left|h(x_0)\int_{\R^3}\left[\frac{1}{|x|}*\left(\widetilde{h}_n(x)\widetilde{u}_n^2-h(x_0)\widetilde{u}^2\right) \right]\widetilde{u}\varphi dx\right|.
\end{aligned}
\end{equation}
Thus, using \eqref{58}--\eqref{510} and $(H_1)$, we have
	$$\begin{aligned}
&\int_{\R^3}\left[\frac{1}{|x|}*h(\varepsilon_n x+\varepsilon_ny_n)\widetilde{u}_n^2\right]h(\varepsilon_n x+\varepsilon_ny_n)\widetilde{u}_n\varphi dx\\
=&h(x_0)^2\int_{\R^3}\left[\frac{1}{|x|}*\widetilde{u}^2\right]\widetilde{u}\varphi dx+o_n(1)\\
=&o_n(1),
\end{aligned}$$
as desired. The rest proof is closely similar to that of Theorem \ref{the1}.
\end{proof}
 \vskip2mm
{\section{Nonexistence of ground state solution}\label{sec6}}
\setcounter{equation}{0}
\vskip2mm 
 In this section, we will study the nonexistence of ground state solutions  for the problem \eqref{1}. The proof of the case $h(x)\equiv 1$ is similar to that of $h(x)$ satisfying condition $(H_1)$, so we just give the proof for the latter case.
 \vskip2mm 
\begin{lemma}\label{L61}
	Assume that $(f_1)$--$(f_5)$ hold. Then, $c_\varepsilon=c_\infty$, for any $\varepsilon>0$.
\end{lemma}
\begin{proof}
 By Lemma \ref{L27}, we see that there exists $u_\infty\in \mathcal{N}_{\infty}$ such that $\mathcal{I}_{\infty}(u_\infty)=c_\infty$.  Let us define $u_\infty^{\zeta _n}=u_\infty(x-\zeta_n)$, where $\zeta_n\in\R^3$ and $|\zeta_n|\to\infty$ as $n\to\infty$. From Lemma \ref{L23} (i) and (ii), there exists a positive bounded sequence $\{t_n\}$  such that $\{t_nu_\infty^{\zeta _n}\}\subset\mathcal{N}_\varepsilon$. Then, we have 
	\begin{equation*}
	\begin{aligned}
		&\hspace{1ex}t_n^2\left(||u_\infty||_{\mathcal{D}^{1,2}}^2+\int_{ \R^3}V(\varepsilon(x+\zeta_n))|u_\infty|^2dx\right)\\
		&+t_n^4\int_{ \R^3}\left[\frac{1}{|x|}*h(\varepsilon(x+\zeta_n)|u_\infty|^2\right]h(\varepsilon(x+\zeta_n)|u_\infty|^2 dx\\
		 =& \sum\limits_{i=1}^{m} t_n^{q_i}\int_{\R^3}Q_i (\varepsilon(x+\zeta_n))|u_\infty|^{q_i}dx + t_n^6
	\int_{ \R^3} K(\varepsilon(x+\zeta_n))|u_\infty|^{6}dx.
	\end{aligned}
	\end{equation*}
Up to a subsequence, we may assume that $t_n\to {t}_0 >0$, as $n\to\infty$. Taking the limit as $n\to\infty$ in the above equality and using condition $(f_5)$, we get
\begin{equation}\label{61}
\begin{aligned}
	t_0^2\left(||u_\infty||_{\mathcal{D}^{1,2}}^2+\int_{ \R^3}V_\infty|u_\infty|^2dx\right) 
= \sum\limits_{i=1}^{m} t_0^{q_i}\int_{ \R^3}Q_i ^\infty|u_\infty|^{q_i}dx + t_0^6
\int_{ \R^3} K^\infty|u_\infty|^{6}dx.
\end{aligned}
\end{equation}	
Note that $u_\infty\in \mathcal{N}_{\infty}$, we have
\begin{equation}\label{62}
	\left(||u_\infty||_{\mathcal{D}^{1,2}}^2+\int_{ \R^3}V_\infty|u_\infty|^2dx\right) 
= \sum\limits_{i=1}^{m}\int_{\R^3}Q_i ^\infty|u_\infty|^{q_i}dx + 
\int_{ \R^3} K^\infty|u_\infty|^{6}dx.
\end{equation}
Combining \eqref{61} and \eqref{62}, we can obtain that $\lim\limits_{n\to\infty}t_n=t_0=1$.
	It follows from $V(x)$ is bounded that there exists a constant $M$ sucht that $|V(\varepsilon(x+\zeta_n))-V_\infty|\le M$. Since $u_\infty\in H_\varepsilon$, then for any $\sigma>0$, there exists $R>0$ such that $|u_\infty|_{2,B_R^C(0)}<\frac{\sigma}{2M}$. Thus, we have
	\begin{equation*}
	\int_{\R^3\backslash B_{R}(0)}|V(\varepsilon(x+\zeta_n))-V_\infty|u_{\infty}^2dx<\frac{\sigma}{2},
	\end{equation*}
	and using the condition $(f_5)$ and $|\zeta_n|\to\infty$, as $n\to\infty$, we have
	\begin{equation*}
	\int_{B_{R}(0)}|V(\varepsilon(x+\zeta_n))-V_\infty|u_{\infty}^2dx<\frac{\sigma}{2},
	\end{equation*}
	as $n\to\infty$.
	Which implies that
	\begin{equation}\label{63}
	\int_{\R^3}V(\varepsilon(x+\zeta_n))u_{\infty}^2dx\to\int_{ \R^3}V_\infty u_{\infty}^2dx,
	\end{equation}
	as $n\to\infty$. Similarly, we can obtain that
	\begin{equation}
		\int_{ \R^3}K(\varepsilon(x+\zeta_n))|u_{\infty}|^6dx\to\int_{ \R^3}K^\infty|u_{\infty}|^6dx,
	\end{equation}
	and
	\begin{equation}\label{65}
		\int_{ \R^3}Q_i(\varepsilon(x+\zeta_n))|u_{\infty}|^{q_i}dx\to\int_{\R^3}Q_i^\infty|u_{\infty}|^{q_i}dx,
	\end{equation}
	as $n\to\infty$. From \eqref{63}--\eqref{65} and $\lim\limits_{n\to\infty}t_n=1$, we can deduce that
	\begin{equation}\label{66}
	\begin{aligned}
	c_\varepsilon=&\inf\limits_{u\in \mathcal{N}_\varepsilon} \mathcal{I}_\varepsilon(u)\\
	\le& \mathcal{I}_\varepsilon(t_nu_\infty^{\zeta _n}) \\
	=&\mathcal{I}_\infty(t_nu_\infty)+\frac{t_n^2}{2}	
	\int_{ \R^3}\left(V(\varepsilon(x+\zeta_n))-V_\infty\right)u_{\infty}^2dx\\
	&+\frac{t_n^4}{4}\int_{ \R^3}\left[\frac{1}{|x|}*h(\varepsilon(x+\zeta_n)u_\infty^2\right]h(\varepsilon(x+\zeta_n)u_\infty^2 dx\\
	& - \sum\limits_{i=1}^{m}\frac{t_n^{q_i}}{q_i} \int_{\R^3}(Q_i (\varepsilon(x+\zeta_n))-Q_i^\infty)|u_\infty|^{q_i}dx - \frac{t_n^6}{6}
	\int_{ \R^3} (K(\varepsilon(x+\zeta_n))-K^\infty)|u_\infty|^{6}dx\\
	\to&\mathcal{I}_{\infty}(u_\infty)=c_\infty,
	\end{aligned}
	\end{equation}
	as $n\to\infty$. Thus, we get $c_\varepsilon\le c_\infty$. On the other hand, it follows from condition $(f_5)$ that $\mathcal{I}_{\infty}(u)\le \mathcal{I}_{\varepsilon}(u)$, for any $u\in H_\varepsilon$. Then, for any $u\in\mathcal{N}_\infty$, by the similar arguments to Lemma \ref{L23} (i) and (ii), we can obtain that there exists a unique bounded $t_\varepsilon>0$ such that $t_\varepsilon u\in\mathcal{N}_\varepsilon$ and $\mathcal{I}_{\varepsilon}(t_{\varepsilon}u):=\max\limits_{t\geq0} \mathcal{I}_{\varepsilon}(tu)$. So it is easy to see that
	\begin{equation}\label{67}
	c_\infty
	=\inf\limits_{u\in \mathcal{N}_\infty} \mathcal{I}_\infty(u)
	\le\inf\limits_{u\in \mathcal{N}_\infty} \mathcal{I}_\varepsilon(u)
	\le\inf\limits_{u\in \mathcal{N}_\infty} \mathcal{I}_\varepsilon(t_\varepsilon u)
	=\inf\limits_{u\in \mathcal{N}_\varepsilon} \mathcal{I}_\varepsilon(u)=c_\varepsilon.
	\end{equation}
	Combining \eqref{66} and \eqref{67}, the proof is complete.
\end{proof}
\begin{proof}[\rm{}\textbf{Proof of Theorem \ref{the3}}] Arguing by contradiction, suppose that problem \eqref{1} has a positive ground state solutions, i.e., there exist $\varepsilon_0>0$ and $u_0\in\mathcal{N}_{\varepsilon_0}$ such that $\mathcal{I}_{\varepsilon_0}(u_0)=c_{\varepsilon_0}$. By the similar arguments to Lemma \ref{L23} (i), there exists  $t_{\varepsilon_0}>0$ such that $t_{\varepsilon_0} u_0\in\mathcal{N}_\infty$. From condition $(f_5)$, we get that  $\mathcal{I}_{\infty}(u)\le \mathcal{I}_{\varepsilon_0}(u)$, for any $u\in H_{\varepsilon_0}$. Then, we have
\begin{equation}
c_\infty
=\inf\limits_{u\in \mathcal{N}_\infty} \mathcal{I}_\infty(u)
\le\mathcal{I}_\infty(t_{\varepsilon_0}u_0)
\le \mathcal{I}_{\varepsilon_0}(t_{\varepsilon_0}u_0)
\le\mathcal{I}_{\varepsilon_0}(u_0) =c_{\varepsilon_0},
\end{equation}
it follows from Lemma \ref{L61} that 
$\mathcal{I}_{\varepsilon_0}(t_{\varepsilon_0}u_0)
=\mathcal{I}_{\infty}(t_{\varepsilon_0}u_0)$. On the other hand, we have
\begin{equation}
\begin{aligned}
\mathcal{I}_{\varepsilon_0}(t_{\varepsilon_0}u_0)
=&\mathcal{I}_{\infty}(t_{\varepsilon_0}u_0)+\frac{t_{\varepsilon_0}^2}{2}	
\int_{ \R^3}\left(V(\varepsilon_0x)-V_\infty\right)u_{\infty}^2dx\\
&+\frac{t_{\varepsilon_0}^4}{4}\int_{ \R^3}\left[\frac{1}{|x|}*h(\varepsilon_0x)u_\infty^2\right]h(\varepsilon_0x)u_\infty^2 dx + \sum\limits_{i=1}^{m}\frac{t_{\varepsilon_0}^{q_i}}{q_i} \int_{ \R^3}(Q_i^\infty-Q_i (\varepsilon_0x))|u_\infty|^{q_i}dx\\
& + \frac{t_{\varepsilon_0}^6}{6}
\int_{ \R^3} (K^\infty - K(\varepsilon_0x))|u_\infty|^{6}dx,
\end{aligned}
\end{equation}
which implies that $\mathcal{I}_{\varepsilon_0}(t_{\varepsilon_0}u_0)
>\mathcal{I}_{\infty}(t_{\varepsilon_0}u_0)$. This is a contradiction and ends the proof.
\end{proof}
\par\noindent
\textbf{Acknowledgments}
\vskip2mm
This work is  supported by National Natural Science Foundation of China (No. 11671403).

{}
\end{document}